\documentclass[12pt,a4paper]{article}
 \usepackage{emlines2,amsmath,amstext,amssymb,amscd}

\newcommand{\Xcomment}[1]{}

 \makeatletter
\renewcommand{\section}{\@startsection{section}{1}{0pt}%
{-3.5ex plus -1ex minus -.2ex}{2.3ex plus .2ex}%
{\normalfont\Large}}
 \makeatother

 \makeatletter
\renewcommand{\subsection}{\@startsection{subsection}{2}{0pt}%
{-3.0ex plus -1ex minus -.2ex}{1.5ex plus .2ex}%
{\normalfont\normalsize\bf}}
 \makeatother

\newtheorem{theorem}{Theorem}[section]
\newtheorem{lemma}[theorem]{Lemma}
\newtheorem{corollary}[theorem]{Corollary}

\newtheorem{prop}[theorem]{Proposition}

\newenvironment{proof}{\noindent{\it Proof.\/}}%
{$\qed$\vskip 0.1in}

\newenvironment{numitem}{\refstepcounter{equation}\begin{enumerate}%
\item[(\arabic{equation})]$\quad$}{\end{enumerate}}

\newcommand{\refeq}[1]{(\ref{eq:#1})}  

\def\rest#1{_{\,\vrule height 1.8ex width 0.05em depth 0pt\, #1}}
\def\qed{ \ \vrule width.2cm height.2cm depth0cm}

\def\tilde{\widetilde}
\def\hat{\widehat}
\def\bar{\overline}
\def\eps{\epsilon}

\def\suppo{{\rm supp}}

\def\Rset{{\mathbb R}}
\def\Zset{{\mathbb Z}}

\def\Bscr{{\cal B}}

\def\Dscr{{\cal D}}
\def\Fscr{{\cal F}}
\def\Gscr{{\cal G}}
\def\Iscr{{\cal I}}
\def\Lscr{{\cal L}}
\def\Mscr{{\cal M}}
\def\Pscr{{\cal P}}
\def\Qscr{{\cal Q}}
\def\Rscr{{\cal R}}
\def\Sscr{{\cal S}}
\def\Tscr{{\cal T}}
\def\Wscr{{\cal W}}
\def\Xscr{{\cal X}}

\def\bfzero{{\bf 0}}

\def\Xdown{X_{\downarrow}}
\def\Xup{X^{\uparrow}}

\begin{document}

 \begin{center}
{\Large\bf On bases of tropical Pl\"ucker functions%
\footnote[1]{This research was supported in part by NWO--RFBR grant
047.011.2004.017 and by RFBR grant 05-01-02805 CNRSL\_a.}}
 \end{center}

 \begin{center}
{\sc Vladimir~I.~Danilov}\footnote[2] {Central Institute of Economics and
Mathematics of the RAS, 47, Nakhimovskii Prospect, 117418 Moscow, Russia;
email: danilov@cemi.rssi.ru.},
{\sc Alexander~V.~Karzanov}\footnote[3]{Institute for System Analysis of
the RAS, 9, Prospect 60 Let Oktyabrya, 117312 Moscow, Russia; email:
sasha@cs.isa.ru. A part of this paper was written when this author was
visiting Equipe Combinatoire et Optimisation, Univ. Paris-6 and supported
by a grant from CNRS, France.},
{\sc and Gleb~A.~Koshevoy}\footnote[4]{Central Institute of Economics and
Mathematics of the RAS, 47, Nakhimovskii Prospect, 117418 Moscow, Russia;
email: koshevoy@cemi.rssi.ru}
\end{center}


\begin{quote}
{\bf Abstract.} \small We consider functions $f:B\to\Rset$ that obey
tropical analogs of classical Pl\"ucker relations on minors of a matrix.
The most general set $B$ that we deal with in this paper is of the form
$\{x\in \Zset^n\colon 0\le x\le a,\; m\le x_1+\ldots+x_n\le m'\}$ (a
rectangular integer box ``truncated from below and above''). We construct
a basis for the set $\Tscr$ of tropical Pl\"ucker functions on $B$, a
subset $\Bscr\subseteq B$ such that the restriction map
$\Tscr\to\Rset^\Bscr$ is bijective. Also we characterize, in terms of the
restriction to the basis, the classes of submodular, so-called
skew-submodular, and discrete concave functions in $\Tscr$, discuss a
tropical analogue of the Laurentness property, and present other results.

 \medskip
{\em Keywords}\,: Pl\"ucker relation, tropicalization, octahedron
recurrence, submodular function, rhombic tiling,
Laurent phenomenon

\medskip
{\em AMS Subject Classification}\, 05C75, 05E99
  \end{quote}

\parskip=3pt


\section{Introduction}  \label{sec:intr}

There are nice algebraic relations on minors of a matrix that have been
established long ago. For a positive integer $n$, let $[n]$ denote the
ordered set of $n$ elements $1,2,\ldots,n$. For an $n\times n$ matrix $M$
and a set $J\subseteq [n]$, let $\Delta_J$ denote the determinant of the
submatrix of $M$ formed by the column set $J$ and the row set
$\{1,\ldots,|J|\}$. Then: (i) for any triple $i<j<k$ in $[n]$ and
$X\subseteq [n]-\{i,j,k\}$,
  $$
\Delta_{Xik}\Delta_{Xj}=\Delta_{Xij}\Delta_{Xk}+\Delta_{Xi}\Delta_{Xjk};
  $$
and (ii) for any quadruple $i<j<k<\ell$ in $[n]$ and $X\subseteq
[n]-\{i,j,k,\ell\}$,
  $$
\Delta_{Xik}\Delta_{Xj\ell}=\Delta_{Xij}\Delta_{Xk\ell}
  +\Delta_{Xi\ell}\Delta_{Xjk},
  $$
where for brevity we write $Xi'\ldots j'$ instead of
$X\cup\{i'\}\cup\ldots\cup\{j'\}$. These equalities represent simplest
cases of so-called {\em Pl\"ucker's relations}. (About classical
Pl\"ucker's type quadratic relations involving flag minors of a matrix,
see, e.g.,~\cite{Fu}).

Relations as above can be stated in an abstract form; namely, one can
consider a function $g$ on the {\em Boolean (hyper)cube} $\{0,1\}^{[n]}$
or on an appropriate part of it and impose the condition
         $$
  g(Xik)g(Xj)=g(Xij)g(Xk)+g(Xi)g(Xjk),
      $$
and/or
   $$
 g(Xik)g(Xjl)=g(Xij)g(Xk\ell)+g(Xi\ell)g(Xjk),
   $$
for $X,i,j,k,\ell$ as above (identifying a subset of $[n]$ with the
corresponding 0,1 vector). Such a $g$ is said to be an {\em algebraic
Pl\"ucker function}, or an {\em AP-function}. For simplicity, in what
follows we will deal with only real-valued functions.

Tropical analogs of these relations, coming up when multiplication is
replaced by addition and addition is replaced by taking maximum, are
viewed as
  \begin{equation}  \label{eq:P3}
 f(Xik)+f(Xj)=\max\{f(Xij)+f(Xk),\; f(Xi)+f(Xjk)\},
  \end{equation}
and
  \begin{equation}  \label{eq:P4}
 f(Xik)+f(Xjl)=\max\{f(Xij)+f(Xk\ell),\; f(Xi\ell)+f(Xjk)\},
  \end{equation}
(see, e.g.,~\cite[Sec.~2]{BFZ}), and a function $f$ obeying~\refeq{P3}
and~\refeq{P4} is said to be a {\em tropical Pl\"ucker function}.

In this paper we do not restrict ourselves by merely the Boolean case, but
consider functions defined on a set of a more general form, namely, on a
box or a truncated box in $\Zset^{[n]}$, and satisfying natural
generalizations of~\refeq{P3} and~\refeq{P4} to such domains.

More precisely, let $a=(a_1 ,...,a_n )$ be an $n$-tuple of natural
numbers; we refer to $|a|:=a_1+\ldots+a_n$ as the {\em size} of $a$. The
$a$-{\em box} is the set $B(a)$ of integer vectors $x=(x_1,...,x_n)$
satisfying the box constraints $0\le x_i \le a_i$ for all $i\in [n]$. In
particular, the Boolean cube $2^{[n]}$ is just the box $B({\bf 1})$, where
${\bf 1}$ is the all-unit vector. Given integers $0\le m\le m'\le |a|$, by
the {\em truncated box} bounded by $m,m'$ we mean the subset of vectors
$x\in B(a)$ with $m\le|x| \le m'$. It is denoted by $B_m^{m'}(a)$ and the
number $m'-m$ is regarded as its {\em width}. So $B(a)=B_0^{|a|}(a)$. For
$m\le p\le m'$, the $p$th {\em layer} of $B_m^{m'}(a)$ is formed by the
vectors of size $p$ in it. Sometimes we will deal with a shifted box
$B(a'|\,a''):=\{x\in\Zset^n\colon a'\le x\le a''\}$, where
$a',a''\in\Zset^n$ and $a'\le a''$.

Three special cases will be important to us. When $a$ is all-unit, we
obtain the {\em truncated Boolean cube} $B_m^{m'}({\bf 1})$. When $m=m'$,
we obtain a truncated box with zero width; it may be denoted as $B_m(a)$
and called a {\em slice}. When, in addition, $a={\bf 1}$, the slice turns
into the {\em hyper-simplex} $\{S\subseteq [n]\colon |S|=m\}$.

 \smallskip
{\bf Definition.} Let $B$ be a truncated box $B_m^{m'}(a)$. A function
$f:B\to\Rset$ is called a {\em tropical Pl\"ucker function}, or a {\em
TP-function} for short, if it satisfies the {\em TP3-relation}:
  \begin{eqnarray}
 &&f(x+1_i+1_k)+f(x+1_j)   \label{eq:GP3} \\
 &&\qquad\qquad   =\max\{f(x+1_i+1_j)+f(x+1_k),\; f(x+1_i)+f(x+1_j+1_k)\}
       \nonumber
  \end{eqnarray}
for any $x$ and $1\le i<j<k\le n$, and satisfies the {\em TP4-relation}:
  \begin{eqnarray}
 &&f(x+1_{i}+1_{k})+f(x+1_{j}+1_{\ell})  \label{eq:GP4} \\
 &&=\max\{f(x+1_{i}+1_{j})+f(x+1_{k}+1_{\ell}),\;
 f(x+1_{i}+1_{\ell})+f(x+1_{j}+1_{k})\}
 \nonumber
  \end{eqnarray}
for any $x$ and $1\le i<j<k<\ell\le n$, provided that all six vectors
occurring as arguments in~\refeq{GP3} belong to $B$, and similarly
for~\refeq{GP4}. Here $1_q$ denotes $q$th unit base vector in
$\Rset^{[n]}$.

 \smallskip
So each TP4-relation concerns vectors of the same layer and vanishes in
$p$th layer for $p=0,1,|a|-1,|a|$ or when $n<4$, while each TP3-one
concerns vectors of two neighboring layers. The zero vector (as well as
$a$) occurs in no relation at all and we will often assume that $f({\bf
0})=0$. When $a={\bf 1}$, \refeq{GP3} and~\refeq{GP4} turn into~\refeq{P3}
and~\refeq{P4}, respectively. Sometimes, when all six vectors occurring as
arguments in~\refeq{GP3} belong to $B$, we will refer to $(x,i,j,k)$ as a
{\em (feasible) cortege}, and similarly for $(x,i,j,k,\ell)$
(concerning~\refeq{GP4}).

Functions satisfying algebraic or tropical Pl\"ucker-type relations have
been studied in literature. Such functions on Boolean cubes are considered
by Berenstein, Fomin and Zelevinsky~\cite{BFZ} in connection with the
total positivity and Lusztig's canonical bases; see also~\cite{K1}.
Henriques~\cite{H} considers AP-functions on the set of integer solutions
of the system $0\le x_i \le m-1$, $x_1+\ldots+x_n=m$, and refers to the
work of Fock and Goncharov~\cite{FG} for results on such functions. The
tropical analogs of those AP-functions, with one additional condition (of
rhombic concavity) imposed on them, form a class of polymatroidal concave
functions, or $M$-functions, studied by Murota~\cite{M}; see
also~\cite{La}. Tropical Pl\"ucker functions in dimension 3 and 4 are
considered in~\cite{DK,KTW,Sp} in connection with the so-called octahedron
recurrence. An instance of Pl\"ucker relations is a relation on six
lengths between four horocycles in the hyperbolic plane with distinct
centers at infinity~\cite{PM}. The TP4-functions on a hyper-simplex form
a special case of so-called {\em valuated matroids} introduced by Dress
and Wenzel~\cite{DW}.

 \smallskip
The set of TP-functions on $B=B_m^{m'}(a)$ is denoted by $\Tscr(B)$. In
this paper we give an explicit construction of a basis for the
TP-functions on $B$.

 \smallskip
{\bf Definition.} A subset $\Bscr\subseteq B=B_m^{m'}(a)$ is called a {\em
TP-basis}, or simply a {\em basis}, if the restriction map
$res:\Tscr(B)\to \Rset^{\Bscr}$ is a bijection. In other words, each
TP-function on $B$ is determined by by values on $\Bscr$, and moreover,
values on $\Bscr$ can be chosen arbitrarily.

 \smallskip
Showing that such a basis does exist, we construct a basis of a rather
simple form, as follows.

For a nonzero vector $x\in B(a)$, let $c(x)$ and $d(x)$ denote,
respectively, the first and last elements (w.r.t. the order in $[n]$) in
the support $\suppo(x)=\{i\colon x_i\ne 0\}$ of $x$. We say that $x$ is a
{\em fuzzy-interval}, or, briefly, a {\em fint}, if $x_i=a_i$ for all
$c(x)<i<d(x)$. We say that $x$ is a {\em sesquialteral fuzzy-interval}, or
a {\em sint}, if $x$ is not a fint and is representable as the sum of two
fints $x',x''$ such that $d(x')<c(x'')$, and $x'_i=a_i$ for
$i=1,\ldots,d(x')-1$. When $a={\bf 1}$, a fint turns into an {\em
interval} $\{c,c+1,\ldots,d\}$ in $[n]$, denoted as $[c..d]$, and a sint
turns into a {\em sesquialteral interval}, a set of the form $[1..d_1]\cup
[c_2..d_2]$ with $c_2>d_1+1$.

Let $Int(a;p)$ and $Sint(a;p)$ denote the sets of fints and sints $x$ of
size $|x|=p$ in $B(a)$, respectively. Our main theorem is the following.

\medskip
 \noindent
{\bf Theorem A} {\em The set $\Bscr:=Sint(a;m)\cup Int(a;m)\cup
Int(a;m+1)\cup\ldots\cup I(a;m')$ is a TP-basis for the truncated box
$B_m^{m'}(a)$.}

 \medskip
We call the basis $\Bscr$ figured in this theorem {\em standard}. (When
$m=0$, we default include the zero vector in the basis as well, without
indicating it explicitly.) Observe that the standard basis involves sints
only from the lowest layer. In particular, the set $Int(a;1)\cup\ldots\cup
Int(a;n)$ gives a basis for the box $B(a)$, and $Sint(a;m)\cup Int(a;m)$
gives a basis for the slice $B_m(a)$. Due to this theorem, $\Tscr(B)$ with
$B=B_m^{m'}(a)$ is representable as a $|\Bscr|$-dimensional polyhedral
conic complex (a fan) in $\Rset^{B}$. It contains a large lineal since any
quasi-separable function of the form
$\varphi_1(x_1)+...+\varphi_n(x_n)+\varphi_0(x_1+...+x_n)$ (where
$\varphi_i$ is an arbitrary function in one variable) is a TP-function of
which addition to any other TP-function maintains the TP-relations.

To illustrate the theorem, consider the hyper-simplex $H$ for $n=4$ and
$m=2$; it consists  of six sets, which may be denoted as  $12,\ 13,\ 14,\
23,\ 24,\ 34$. By adding an appropriate quasi-separable function, any
TP-function $f$ on $H$ is transformed so as to take zero value on the
points $12,\ 13,\ 14,\ 24$. Then the unique TP4-relation (concerning $13,\
24$) implies $\max\{f(23),f(34)\}=0$. This means that, modulo the lineal,
the set of TP-function is represented as the union of two rays in
$\Rset^2$, namely, $(\Rset_-,0)$ and $(0,\Rset_-)$, and therefore, it is
piecewise-linear-morphic to the line $\Rset$.

An easy consequence of Theorem~A is that a TP-function $f$ on a truncated
box $B_m^{m'}(a)$ can be extended to a TP-function on the entire box
$B(a)$. Indeed, first we take the restriction of $f$ to the standard basis
for $B_m^{m'}(a)$ and extend it to the standard basis for $B_m^{|a|}(a)$
by assigning arbitrary values on $Int(a;m'+1) \cup\ldots\cup Int(a;|a|)$.
This determines a TP-function $g$ on $B_m^{|a|}(a)$ coinciding with $f$ on
$B_m^{m'}(a)$. Then we consider the {\em complementary} function $g^\ast$
for $g$, which is defined on the vectors $x\in B_0^{|a|-m}(a)$ by
$g^\ast(x):=g(a-x)$; clearly $g^\ast$ is a TP-function again. Finally, we
extend $g^\ast$ into a TP-function $h$ on $B_0^{|a|}(a)=B(a)$ by acting as
at the first step. Then $h^\ast$ is the desired extension of $f$ to
$B(a)$.

Special cases of Theorem~A have appeared in some earlier works. The
corresponding result for Boolean cubes is given in~\cite{BFZ}, and for
hyper-simplexes in~\cite{Po}; see also~\cite{Sc,SW}. The algebraic version
in the case of a ``simplicial'' slice $\{x\in\Zset_+^n \colon \sum
x_i=m\}$ was announced in~\cite{H} with a claim that it could be obtained
by use of results on cluster algebras in~\cite{FG}.

Our proof of Theorem~A is direct and relatively short. It consists of
three stages. At the first stage, we prove that the corresponding
restriction map $res$ is injective. At the second stage, we prove the
surjectivity of $res$ for the Boolean version (Theorem~A$'$ in
Section~\ref{sec:boolean}). At the third stage, we reduce the general case
to the Boolean one. The core of the whole proof is a {\em flow-in-matrix}
method, which consists in representing {\em any} TP-function $f$ on a
truncated Boolean cube by use of maximum weight flows (systems of paths)
in a weighted grid associated with an $|a|\times m'$ matrix whose entries
are determined by the values of $f$ on the intervals and the $m$-sized
sesquialteral intervals in $[n]$.

 \smallskip
Another group of results in this paper concerns an interrelation between
TP-bases and rhombic tilings, and characterizations of special classes of
TP-functions.

Given a basis $\Bscr$ (e.g., the standard one), we can produce more bases
by making a series of elementary transformations relying on TP3- or
TP4-relations, referring to them as {\em mutations}, or {\em flips}. More
precisely, suppose there is a cortege $(x,i,j,k)$ such that the four
vectors occurring in the right hand side of~\refeq{GP3} and one vector $y$
in the left hand side, say, $y=x+1_j$, belong to $\Bscr$. It is easy to
see that the replacement in $\Bscr$ of $y$ by the other vector in the left
hand side, namely, $x+1_i+1_k$, results in a basis as well; we can further
transform the latter basis in a similar way. Analogous transformations via
TP4-relations can be applied to corresponding corteges $(x,i,j,k,\ell)$.

When dealing with an entire box $B(a)$ (in particular, with the cube
$2^{[n]}$), the standard basis, as well as many other (but not all) bases
obtained from it by a series of TP3-mutations, can be associated with a
{\em rhombic tiling diagram} on a $2n$-gone (a zonogon), giving a nice
visualization of the basis. (For rhombic tilings, see, e.g.,~\cite{ER}.)

To illustrate this, consider the cube $2^{[3]}$. The standard basis
$\Bscr$ for it consists of the six intervals $1,2,3,12,23,123$ to which we
also add the empty interval $\{\emptyset\}$. There is only one basis
$\Bscr'$ different from $\Bscr$; it is obtained from $\Bscr$ by the
TP3-mutation $2\rightsquigarrow 13$. The cube and the rhombic tiling
diagrams for $\Bscr$ and $\Bscr'$ are drawn in the picture:

   \begin{center}
  \unitlength=1mm
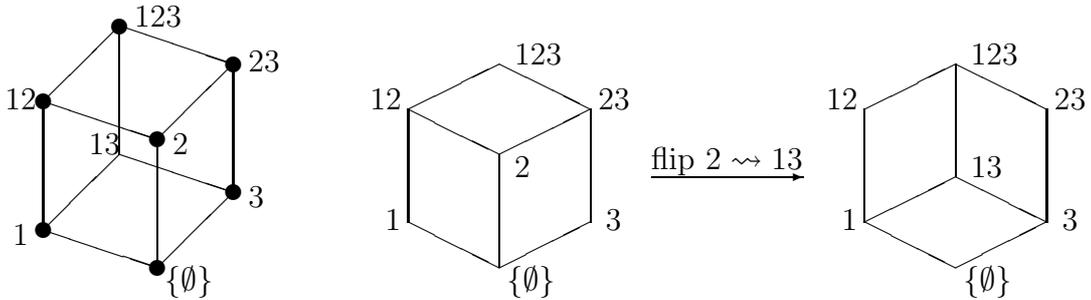
    \begin{picture}(150,40)
  \put(10,10){\line(0,1){17}}
  \put(20,20){\line(0,1){17}}
  \put(25,5){\line(0,1){17}}
  \put(35,15){\line(0,1){17}}
  \put(25,5){\line(1,1){10}}
  \put(10,10){\line(1,1){10}}
  \put(25,22){\line(1,1){10}}
  \put(10,27){\line(1,1){10}}
  \put(10,10){\line(3,-1){15}}
  \put(20,20){\line(3,-1){15}}
  \put(10,27){\line(3,-1){15}}
  \put(20,37){\line(3,-1){15}}
  \put(10,10){\circle*{2}}
  \put(10,27){\circle*{2}}
  \put(25,5){\circle*{2}}
  \put(25,22){\circle*{2}}
  \put(20,37){\circle*{2}}
  \put(35,15){\circle*{2}}
  \put(35,32){\circle*{2}}
  \put(26,2){$\{\emptyset\}$}
  \put(6,8){1}
  \put(37,13){3}
  \put(16,20){13}
  \put(27,20){2}
  \put(5,26){12}
  \put(37,31){23}
  \put(22,37){123}
  \put(70,5){\line(0,1){15}}
  \put(58,11){\line(0,1){15}}
  \put(82,11){\line(0,1){15}}
  \put(70,5){\line(2,1){12}}
  \put(70,20){\line(2,1){12}}
  \put(58,26){\line(2,1){12}}
  \put(58,11){\line(2,-1){12}}
  \put(58,26){\line(2,-1){12}}
  \put(70,32){\line(2,-1){12}}
  \put(71,2){$\{\emptyset\}$}
  \put(55,10){1}
  \put(84,10){3}
  \put(72,17){2}
  \put(53,26){12}
  \put(83,26){23}
  \put(72,32){123}
  \put(90,17){\vector(1,0){20}}
  \put(90,18){flip $2\rightsquigarrow 13$}
  \put(130,17){\line(0,1){15}}
  \put(118,11){\line(0,1){15}}
  \put(142,11){\line(0,1){15}}
  \put(130,5){\line(2,1){12}}
  \put(118,26){\line(2,1){12}}
  \put(118,11){\line(2,1){12}}
  \put(118,11){\line(2,-1){12}}
  \put(130,32){\line(2,-1){12}}
  \put(130,17){\line(2,-1){12}}
  \put(131,2){$\{\emptyset\}$}
  \put(115,10){1}
  \put(144,10){3}
  \put(132,17){13}
  \put(113,26){12}
  \put(143,26){23}
  \put(132,32){123}
  \end{picture}
 \end{center}

We refer to a TP-basis for $B(a)$ that corresponds to a rhombic tiling as
a {\em normal} one. As an additional result, we give necessary and
sufficient conditions on a subset of $B(a)$ that can be extended into a
normal basis, and develop a polynomial-time algorithm to find such a basis
(a similar problem for the class of all bases seems to be more
sophisticated).

Using the correspondence between the normal bases and rhombic tilings, we
then study the classes of {\em submodular} and {\em skew-submodular}
TP-functions $f$ on a box $B(a)$, which means that $f$ satisfies the
inequalities of the form
  $$
f(x+1_i)+f(x+1_j)\ge f(x)+f(x+1_i+1_j)
  $$
in the former case, and of the form
  $$
  f(x+1_i+1_j)+f(x+1_j)\ge f(x+1_i)+f(x+2\cdot 1_j)
  $$
in the latter case. It turns out that each class admits a characterization
in terms of the restriction of $f$ to the standard basis (or even to an
arbitrary normal basis) $\Bscr$. More precisely, we show that for a
TP-function $f$, the above submodular (skew-submodular) inequalities are
propagated by the TP3-recurrence, starting from such inequalities within
$\Bscr$. Furthermore, we explain that when both the submodular and
skew-submodular inequalities hold, the function $f$ possesses the property
of {\em discrete concavity} (more precisely, polymatroidal concavity, in
the sense that for the minimal concave function $g$ on the convex hull of
$B(a)$ such that $g\rest{{B(a)}}\ge f$, all affine regions of $g$ are
generalized polymatroids).

 \smallskip
Finally, returning to the flow-in-matrix method and considering the
TP-functions $f$ on the cube $2^{[n]}$, we show that for each set
$X\subset [n]$ not in the standard basis, the value $f(X)$ is
expressed as a piece-wise linear convex function $h$ (invariant of $f$)
of which arguments are the values on the standard basis.
A similar property is shown to take place in
case of truncated cubes and entire boxes. This behavior of TP-functions
with respect to the standard basis can be regarded as exhibiting
a tropical analogue of the so-called {\em Laurent phenomenon}
(for the Laurent phenomenon under the cube recurrence, see~\cite{FZ,HS}).
Moreover, it turns out that all coefficients in the linear pieces of $h$
are integers between $-1$ and 2. Also there is an interesting
relation between such pieces and special Gelfand-Tsetlin patterns (or
semi-standard Young tableaux).

 \smallskip
This paper is organized as follows. In Section~\ref{sec:inject} we prove
the simpler part of Theorem~A, namely, that the corresponding restriction
map $res$ is injective. Also we explain there a rather surprising fact
that the TP4-relations follow from the TP3-ones unless $m=m'$. The other
part of Theorem~A, concerning the surjectivity of $res$, is more involved.
We prove the surjectivity for the Boolean case in
Section~\ref{sec:boolean}, and for the general case in
Section~\ref{sec:gen}, thus completing the proof of Theorem~A. Relations
between bases, their mutations and rhombic tiling diagrams are discussed
in Section~\ref{sec:tiling}; here we also consider the problem of
extendability of a given subset $X\subset B(a)$ into a normal basis.
Sections~\ref{sec:submod},~\ref{sec:skew} and~\ref{sec:concave} are
devoted, respectively, to submodular, skew-submodular and discrete concave
TP-functions. The concluding Section~\ref{sec:Laurent} discusses the
tropical Laurentness property for TP-functions.

 \smallskip
{\em Acknowledgement.} Gleb Koshevoy thanks IHES (Bures-sur-Yvette,
France) for financial support and hospitality.

\section{Injectivity of the restriction map} \label{sec:inject}

In this section we prove that the restriction map to the standard basis is
injective (which is simpler than the proof of its surjectivity, given
throughout Sections~\ref{sec:boolean} and~\ref{sec:gen}) and also
demonstrate an additional result.

Consider a truncated box $B=B_m^{m'}(a)$ and let $\Bscr:=Sint(a;m)\cup
Int(a;m)\cup Int(a;m+1)\cup\ldots\cup I(a;m')$. As in the Introduction,
given $x\in B$, we denote the first and last elements in the support
$\suppo(x)$ of $x$ by $c(x)$ and $d(x)$, respectively. Also we introduce
the following numbers:
   \begin{eqnarray}
&&\mbox{$\alpha(x)$ is the maximal $i\in[n]$ such that $i<d(x)$
  and $x_i<a_i$}; \label{eq:alphabeta} \\
&& \mbox{$\beta(x)$ is the maximal $i'\in[n]$ such that $i'<\alpha(x)$
  and $x_{i'}>0$};  \nonumber \\
&& \mbox{$\gamma(x)$ is the maximal $i''\in[n]$ such that
   $i''<\beta(x)$ and $x_{i''}<a_{i''}$}; \nonumber
   \end{eqnarray}
Observe that $\beta(x)$ does not exist if and only if $x$ is a fint
(fuzzy-interval), while $\gamma(S)$ does not exist if and only if $x$ is a
fint or a sint (sesquialteral fuzzy-interval).

 \begin{prop} \label{pr:inj_gen}
The restriction map $res:\Tscr(B)\to \Rset^\Bscr$ is injective, i.e., any
TP-function on $B$ is determined by its values within $\Bscr$.
  \end{prop}
  \begin{proof}
Let $f\in\Tscr(B)$ and $x\in B$. When $\beta(x)$ exists, define
  \begin{equation}\label{eq:eta_gen}
  \eta(x):=|a|(\beta(x)+d(x))+x_{\beta(x)}+x_{d(x)},
  \end{equation}
Consider two cases.

  \smallskip
  {\em Case 1}\/: $|x|=m$.
We use induction on $\eta$ to show that $f(x)$ is determined, via
TP4-relations, by the values of $f$ within $Sint(a;m)\cup Int(a;m)$.

If $x$ is a fint or a sint, we are done, so assume this is not the case.
Then all numbers $i:=\gamma(x)$, $j:=\beta(x)$, $k:=\alpha(x)$ and
$\ell:=d(x)$ are well defined, and we have $i<j<k<\ell$. Put
$x':=x-1_j-1_\ell$ and form five vectors $B:=x'+1_i+1_k$, $C:=x'+1_i+1_j$,
$D:=x'+1_k+1_\ell$, $E:=x'+1_i+1_\ell$ and $F:=x'+1_j+1_k$. From the
definitions in~\refeq{alphabeta} it follows that these vectors belong to
$B$ (and have size $m$). By relation~\refeq{GP4} (with $x'$ in place of
$x$), $f(x)$ is computed from the values of $f$ on $B,C,D,E,F$. Also one
can check that each of the latter vectors either is a fint or is a sint or
the value of $\eta$ on it is less than $\eta(x)$.

So we can apply induction on $\eta$ (the inductive process of computing
$f$ on the lowest layer $B_m(a)$ has as a base the family $Sint(a;m)\cup
Int(a;m)$).

  \smallskip
  {\em Case 2}\/: $|x|>m$.
We show that $f(x)$ is determined, via TP3-relations, by the values of
$f$ within $Sint(a;m)\cup Int(a;m)\cup\ldots\cup Int(a;|x|)$. If $x$ is a
fint, we are done, so assume this is not the case. Put $i:=\beta(x)$,
$j:=\alpha(x)$ and $k:=d(x)$; then $i<j<k$. Put $x':=x-1_i-1_k$.
By~\refeq{GP3} (with $x'$ in place of $x$), $f(x)$ is computed via the
values of $f$ on the vectors
  $$
  B:=x'+1_j,\quad C:=x'+1_i+1_j,\quad D:=x'+1_k,\quad
 E:=x'+1_j+1_k,\quad F:=x'+1_i
   $$
(each of which belongs to $B$, in view of~\refeq{alphabeta} and $|x|>m$).
One can check that for each of $B,C,D,E,F$ at least one of the following
is true: it is a fint; it belongs to the preceding layer; the value of
$\eta$ on it is less than $\eta(x)$. So we can apply induction on the
number of the layer and on $\eta$.
  \end{proof}


 \medskip
In the rest of this section we discuss an interrelation between TP3- and
TP4-conditions.

  \Xcomment{
(We will not use such an interrelation in what follows; nevertheless, this
deserves to be mentioned in our study of TP-functions by its own right).
Each TP3-condition involves elements of two neighboring layers, while each
TP4-relation acts within one layer, and in fact, their roles are not
equal: typically the former conditions imply validity of the latter ones,
as the following proposition shows.
  }

  \begin{prop} \label{pr:TP4}
Let $f$ be a function on $B=B_m^{m'}(a)$ and let $m<m'$. Suppose $f$
satisfies all TP3-conditions~\refeq{GP3} on $B$. Then $f$ is a
TP-function, i.e. $f$ satisfies the TP4-conditions~\refeq{GP4} as well.
  \end{prop}
  \begin{proof}
First we show validity of~\refeq{GP4} for a cortege $(x,i,j,k,\ell)$ with
$m<|x|+2\le m'$. We are going to deal with only vectors of the form
$x+1_{i'}$ or $x+1_{i'}+1_{j'}$, where $i',j'\in\{i,j,k,\ell\}$ ($i'\ne
j'$). For this reason and to simplify notation, one may assume, w.l.o.g.,
that $x={\bf 0}$ and $(i,j,k,\ell)=(1,2,3,4)$ (in which case we, in fact,
deal with the truncated Boolean cube $\{S\subset[4]\colon 1\le|S|\le
2\}$). So we have to prove that
  \begin{equation}  \label{eq:1324}
  f(13)+f(24)=\max\{f(12)+f(34),\; f(14)+f(23)\}
  \end{equation}
(where for brevity $qr$ stands for $1_q+1_r$).

We use the following three TP3-relations for $f$:
   \begin{gather}
  f(24)+f(3)=\max\{f(2)+f(34),\; f(4)+f(23)\};   \label{eq:243} \\
  f(13)+f(2)=\max\{f(1)+f(23),\; f(3)+f(12)\};   \label{eq:132} \\
  f(14)+f(2)=\max\{f(1)+f(24),\; f(4)+f(12)\}.   \label{eq:142}
   \end{gather}

Adding $f(12)$ to both sides of~\refeq{243} gives
  $$
  f(24)+f(3)+f(12)=\max\{f(2)+f(34)+f(12),\; f(4)+f(23)+f(12)\}.
  $$
If in each side of this relation we take the maximum of the expression
there and $f(1)+f(23)+f(24)$, then we obtain
  \begin{gather*}
  \max\{f(24)+f(3)+f(12),\; f(1)+f(23)+f(24)\} \\
  =\max\{f(2)+f(34)+f(12),\; f(4)+f(23)+f(12),\; f(1)+f(23)+f(24)\}.
   \end{gather*}
This can be re-written as
  \begin{gather*}
  \max\{f(3)+f(12),\; f(1)+f(23)\}+f(24) \\
 =\max\{f(2)+f(34)+f(12),\; \max\{f(4)+f(12),\; f(1)+f(24)\}+f(23)\}.
   \end{gather*}
The maximum in the left hand side is equal to $f(13)+f(2)$,
by~\refeq{132}, and the interior maximum in the right hand side is equal
to $f(14)+f(2)$, by~\refeq{142}. Therefore, we have
  \begin{gather*}
  f(13)+f(2)+f(24)
 =\max\{f(2)+f(34)+f(12),\; f(14)+f(2)+f(23)\} \\
   =\max\{f(34)+f(12),\; f(14)+f(23)\}+f(2).
   \end{gather*}
Canceling out $f(2)$ in the left and right sides, we obtain the required
equality~\refeq{1324}.

Next, let $|x|+2=m$. Take the complementary function $f^\ast(a-y):=f(y)$,
$y\in B_m^{m'}(a)$, and consider the vectors $\bar
x:=x+1_i+1_j+1_k+1_\ell$ and $\bar x^\ast:=a-\bar x$. Then $f^\ast$ is a
function on the truncated box $B^\ast:=B_{|a|-m'}^{|a|-m}(a)$ satisfying
the TP3-conditions there, and the vector $\bar x^\ast$ is nonnegative and
satisfies $\bar x^\ast+1_i+1_j+1_k+1_\ell\le a$ and
$|a|-m'<|x^\ast|+2=|a|-m$. So we have a situation as in the previous case
(w.r.t. $B^\ast$) and obtain that relation~\refeq{GP4} holds for the
cortege $(\bar x^\ast,i,j,k,\ell)$. This implies that~\refeq{GP4} holds
for $f$ and $(x,i,j,k,\ell)$.
  \end{proof}

Thus, in the definition of a TP-function, imposing the TP4-relations is
important only when we deal with a slice $B_m(a)$.

\section{Surjectivity in the Boolean case} \label{sec:boolean}

In this section we prove the Boolean version of Theorem~A.

To distinguish between the Boolean and general cases, we modify notation
as follows. Let $0\le m\le m'\le n$. The
parameter $n$ will be fixed throughout and we will usually omit it in
notation for basic objects. We denote by $C_m^{m'}$ the truncated Boolean
cube $\{S\subseteq [n]\colon m\le |S|\le m'\}$, and by $C_p$ the
hyper-simplex consisting of the subsets $S$ of size $|S|=p$ (or the $p$th
layer of $C_m^{m'}$ when $m\le p\le m'$).

Any set $S\subseteq [n]$ is uniquely represented as the union of intervals
$I_1=[c_1..d_1],\ldots,I_r=[c_r..d_r]$ such that $c_j>d_{j-1}+1$ for
$j=2,\ldots,r$; such a representation is denoted as
  $$
  S=I_1\sqcup\ldots \sqcup I_r.
  $$
Recall that a sesquialteral interval, or a {\em sesqui-interval}, is a set
of the form $[1..d_1]\sqcup[c_2..d_2]$ with $d_1+1<c_2$.

We denote the set of $p$-element intervals in $[n]$ by $\Iscr_p$, and the
set of $p$-element sesqui-intervals by $\Sscr_p$. Then the Boolean version
of Theorem~A is the following

 \medskip
 \noindent
{\bf Theorem~A${\bf '}$} {\em Let
$\Bscr:=\Sscr_m\cup\Iscr_m\cup\Iscr_{m+1}\cup\ldots\cup\Iscr_{m'}$ and let
$\rho:\Tscr(C_m^{m'})\to \Rset^\Bscr$ be the restriction map. Then $\rho$
is a bijection, i.e., $\Bscr$ forms a TP-basis for the truncated cube
$C_m^{m'}$.
 }

 \medskip
 \noindent
(Note that $\Sscr_m$ vanishes if $m<2$. Also when $m=0$, we default assume
that $f(\emptyset)=0$ for any $f\in \Tscr(C_0^{m'})$).

It suffices to prove that $\rho$ in this theorem is surjective, as its
injectivity has been shown in the previous section. The proof consists of
several steps and is given throughout the subsections below. It is based
on a method of generating any TP-function on $C_m^{m'}$ by use of a
certain flow model, which we call the {\em flow-in-matrix method}. This
method has as a source a construction of examples of tropical Pl\"ucker
functions in~\cite{BFZ}.

\subsection{Grids, matrices and flows} \label{ssec:grid}

By the {\em grid} of size $n\times m'$ we mean the following directed
graph $\Gamma=\Gamma_{n,m'}$. The vertex set $V_{n,m'}$ of $\Gamma$
consists of elements $v_{pq}$ for $p=1,\ldots,n$ and $q=1,\ldots,m'$. The
edge set $E_{n,m'}$ of $\Gamma$ consists of the pairs $(v_{pq},v_{p'q'})$
such that either $p'=p-1$ and $q'=q$, or $p'=p$ and $q'=q+1$. We visualize
the grid by identifying a vertex $v_{pq}$ with the point $(p,q)$ in the
plane. Then the vertices $v_{11},\ldots,v_{1,n}$ are located in the
bottommost horizontal line of $\Gamma$; we call them the {\em sources} and
denote by $s_1,\ldots,s_{n}$, respectively. The vertices
$v_{11},\ldots,v_{m',1}$, located in the leftmost vertical line of
$\Gamma$, are called the {\em sinks} and denoted by $t_1,\ldots,t_{m'}$,
respectively. The grid $\Gamma_{5,3}$ is illustrated in the picture.

   \begin{center}
  \unitlength=1mm
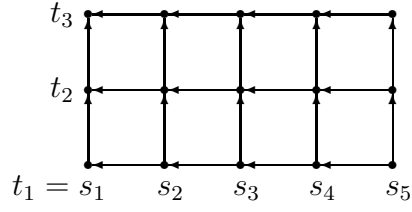
    \begin{picture}(55,30)
 \multiput(10,5)(0,10){3}%
   {\multiput(0,0)(10,0){5}%
     {\put(0,0){\circle*{1}}}
        }
 \multiput(10,5)(0,10){3}%
   {\multiput(0,0)(10,0){4}%
    {\put(10,0){\vector(-1,0){9.5}}}
        }
 \multiput(10,5)(0,10){2}%
   {\multiput(0,0)(10,0){5}%
    {\put(0,0){\vector(0,1){9.5}}}
        }
 \put(0,1){$t_1=s_1$}
 \put(19,1){$s_2$}
 \put(29,1){$s_3$}
 \put(39,1){$s_4$}
 \put(49,1){$s_5$}
 \put(5,14){$t_2$}
 \put(5,24){$t_3$}
  \end{picture}
 \end{center}

By a {\em flow} we mean a collection $\Fscr$ of paths in $\Gamma$, each
path beginning at a source and ending at a sink. We say that $\Fscr$ is
{\em admissible} if:

(i) the paths in $\Fscr$ are pairwise (vertex) disjoint; and

(ii) the sinks occurring in $\Fscr$ are $t_1,\ldots,t_{|\Fscr|}$.

Consider a weighting $w:V_{n,m'}\to\Rset$ on the vertices.  The weight
$w(P)$ of a path $P$ is defined to be the sum of weights $w(v)$ of the
vertices $v$ of $P$, and the weight $w(\Fscr)$ of a flow $\Fscr$ is
$\sum(w(P)\colon P\in\Fscr)$. For a set $S\subseteq [n]$ with $|S|\le m'$,
define
   \begin{equation} \label{eq:f-w}
  f_w(S):=\max\{w(\Fscr)\},
   \end{equation}
where the maximum is taken over all admissible flows $\Fscr$ in
$\Gamma_{n,m'}$ beginning at the set $\{s_p\colon p\in S\}$.

In what follows we will identify the weighting $w$ with the $n\times m'$
matrix $W=(w_{pq})$, where $w_{pq}=w(v_{pq})$. To be consistent with the
visualization of the grid, we should think of $n$ as the number of columns
of $W$, use the first index just for the columns, and assume that $w_{11}$
is the leftmost and bottommost element of the matrix.

The following assertion plays the key role in our proof.

 \begin{theorem} \label{tm:flow}
Let $W=(w_{pq})$ be a real $n\times m'$ matrix. Then the function $f_w$
defined by~\refeq{f-w} on the sets $S\in C_m^{m'}$ is a TP-function.
  \end{theorem}

We denote the map $W\mapsto f_w$ in this proposition by $\phi$. It should
be noted that $\phi$ is not injective in general, that is, one and the
same function $f$ may be derived from several matrices. In order to get a
one-to-one correspondence, we will restrict the set of matrices by taking
as $W$ a matrix that is the sum of two $n\times m'$ matrices $W',W''$,
where
  \begin{numitem}
$W'$ is such that $w'_{pq}=0$ for all $p,q$ with $p<\max\{q,m+1\}$;
  \label{eq:Wp}
  \end{numitem}
  \begin{numitem}
$W''$ is such that:
  \begin{itemize}
 \vspace{-5pt}
 \item[(i)] $w''_{pq}=M$ for $p=2,\ldots,m$ and $q<p$;
 \item[(ii)] $w''_{11}=-\frac{m(m-1)}{2}M$;
 \item[(iii)] $w''_{pq}=0$ otherwise.
   \end{itemize}
   \label{eq:Wpp}
   \end{numitem}
Here $M$ is a sufficiently large positive number w.r.t. the entries of
$W'$ (one can take $M:=nm\max\{|w'_{pq}|\}$). The purpose of adding the
matrix $W''$ to $W'$ will be clear later. The behavior of $W',W''$ is
illustrated in the picture.

   \begin{center}
  \unitlength=1mm
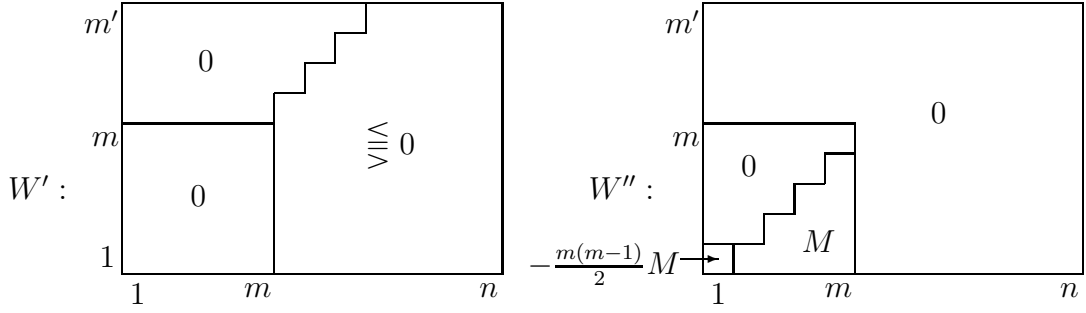
    \begin{picture}(140,45)(0,5)
       \begin{picture}(60,50)(-5,0)
   \put(10,10){\line(1,0){50}}
   \put(10,30){\line(1,0){20}}
   \put(10,46){\line(1,0){50}}
   \put(30,34){\line(1,0){4}}
   \put(34,38){\line(1,0){4}}
   \put(38,42){\line(1,0){4}}
   \put(10,10){\line(0,1){36}}
   \put(30,10){\line(0,1){24}}
   \put(34,34){\line(0,1){4}}
   \put(38,38){\line(0,1){4}}
   \put(42,42){\line(0,1){4}}
   \put(60,10){\line(0,1){36}}
   \put(-5,20){$W':$}
   \put(11,6){$1$}
   \put(26,7){$m$}
   \put(57,7){$n$}
   \put(7,11){$1$}
   \put(6,27){$m$}
   \put(5,42){$m'$}
   \put(19,19){$0$}
   \put(20,37){$0$}
   \put(42,26){$\lesseqqgtr 0$}
         \end{picture}
   \begin{picture}(60,50)(-20,0)
   \put(10,10){\line(1,0){50}}
   \put(10,30){\line(1,0){20}}
   \put(10,46){\line(1,0){50}}
   \put(10,14){\line(1,0){8}}
   \put(18,18){\line(1,0){4}}
   \put(22,22){\line(1,0){4}}
   \put(26,26){\line(1,0){4}}
   \put(10,10){\line(0,1){36}}
   \put(30,10){\line(0,1){20}}
   \put(14,10){\line(0,1){4}}
   \put(18,14){\line(0,1){4}}
   \put(22,18){\line(0,1){4}}
   \put(26,22){\line(0,1){4}}
   \put(60,10){\line(0,1){36}}
   \put(-5,20){$W'':$}
   \put(11,6){$1$}
   \put(26,7){$m$}
   \put(57,7){$n$}
   \put(6,27){$m$}
   \put(5,42){$m'$}
   \put(15,23){$0$}
   \put(40,30){$0$}
   \put(23,13){$M$}
   \put(-13,10){$-\frac{m(m-1)}{2}M$}
   \put(7,12){\vector(1,0){5}}

 \end{picture}

  \end{picture}
 \end{center}

(Note that when $m<2$, $W''$ becomes the zero matrix and $W=W'$. When
$m=0$ and $m'=n$, the essential part of $W'$ is a triangular matrix. When
$m=m'$, the essential part of $W'$ is an $(n-m)\times n$ matrix.) We
denote the sets of matrices $W$ and $W'$ as above by $\Wscr=\Wscr(m,m')$
and $W'=\Wscr'(m,m')$, respectively ($W$ is considered up to the choice of
$M$).

We say that a function $f$ on $C_m^{m'}$ or on $\Bscr$ is {\em normalized}
if $f([1..m])=0$. The set of normalized functions on $C_m^{m'}$ is denoted
by $\Tscr^0(C_m^{m'})$, and we denote by $\Bscr^0$ the set $\Bscr$ from
which the interval $\{[1..m]\}$ is removed. (Any TP-function can be
considered up to adding a constant; so in Theorem~A$'$ one can consider
only normalized TP-functions and their restrictions to $\Bscr$.) The
importance of~\refeq{Wp},\refeq{Wpp} is emphasized by the following
  \begin{prop} \label{pr:matr}
For each normalized function $g:\Bscr\to\Rset$, there exists a matrix
$W'\in\Wscr'(m,m')$ such that $g(S)=f_w(S)$ holds for all $S\in\Bscr$,
where the weighting $w$ corresponds to $W=W'+W''$. Moreover, $W'$ is
unique and the correspondence of $g$ and $W'$ gives a bijection between
$\Rset^{\Bscr^0}$ and $\Wscr'(m,m')$, or between $\Rset^{\Bscr^0}$ and
$\Wscr(m,m')$ (considering the matrices in $\Wscr(m,m')$ up to the choice
of $M$)).
  \end{prop}

We denote the map $g\mapsto W$ in this proposition by $\mu$.

Summing up Theorem~\ref{tm:flow} and Proposition~\ref{pr:matr}, we can
conclude that the restriction map $\rho$ is surjective. Indeed, for each
normalized function $g$ on $\Bscr$, take the matrix $W=\mu(g)$ and form
the function $f=f_w$ ($=\phi(W)$) on $C_m^{m'}$. Then $f$ is a TP-function
whose restriction to $\Bscr$ is just $g$. In fact, we have three
bijections.

 \begin{corollary} \label{cor:3biject}
The maps $\rho,\mu,\phi$ determine bijections between $\Tscr^0(C_m^{m'})$
and $\Rset^{\Bscr^0}$, between $\Rset^{\Bscr^0}$ and $\Wscr(m,m')$, and
between $\Wscr(m,m')$ and $\Tscr^0(C_m^{m'})$, respectively. Their
composition $\phi\circ \mu\circ \rho$ is identical on $\Tscr^0(C_m^{m'})$.
(See the picture.)
  \end{corollary}
   \begin{center}
  \unitlength=1mm
    \begin{picture}(40,20)(0,5)
 \put(5,5){$\Rset^{\Bscr^0}$}
 \put(25,5){$\Wscr$}
 \put(5,20){$\Tscr^0$}
 \put(25,20){$\Tscr^0$}
 \put(16,3){$\mu$}
 \put(2,13){$\rho$}
 \put(28,13){$\phi$}
 \put(15,22){{\rm id}}
 \put(13,6){\vector(1,0){10}}
 \put(5,18){\vector(0,-1){8}}
 \put(27,10){\vector(0,1){8}}
 \put(12,21){\vector(1,0){10}}
  \end{picture}
 \end{center}

Thus, it remains to prove Proposition~\ref{pr:matr} and
Theorem~\ref{tm:flow}.

\subsection{From functions on $\Bscr$ to matrices}
\label{ssec:f-matr}

In this subsection we prove Proposition~\ref{pr:matr}.

For an $n\times m'$ matrix $\tilde W$ and subsets $I\subseteq [n]$ and
$J\subseteq [m']$, let $\tilde w(I\times J)$ denote the weight
$\sum(\tilde w_{pq}\colon p\in I,\ q\in J)$ of the
$I\times J$ submatrix of $\tilde W$.

Let $g$ be a normalized function on $\Bscr$. The desired matrix $W'$ for
$g$ is assigned so as to satisfy the following conditions:
  \begin{equation} \label{eq:for_int}
 g([c..d])=w'([d]\times[d-c+1]) \quad\mbox{for each interval
    $[c..d]\in\Bscr$};
   \end{equation}
  \begin{equation} \label{eq:for_sint}
  g(S)=w'([d_2]\times[d_2-c_2+1]) \quad\mbox{for each sint
    $S=[1..d_1]\sqcup[c_2..d_2]\in\Bscr$}.
   \end{equation}

Subject to~\refeq{Wp}, these conditions determine $W'$ uniquely. To see
this, let $\Pi:=\{(p,q)\colon p=m+1,\ldots,n,\; q=1,\ldots,\min\{p,m'\}\}$
(the set of essential index pairs in~\refeq{Wp}). There is a natural
bijection $\pi:\Pi\to\Bscr^0$, namely:
  \begin{numitem}
  for $(p,q)\in\Pi$,
  \begin{itemize}
 \vspace{-5pt}
 \item[(i)]
  $\pi(p,q)$ is the interval $[p-q+1..p]$ (of size $q$) if $q\ge m$;
 \item[(ii)]
  $\pi(p,q)$ is the sint $[1..m-q]\sqcup[p-q+1..p]$
(of size $m$) if $q<m$.
  \end{itemize}
  \label{eq:corresp}
  \end{numitem}
Now using~\refeq{for_int}--\refeq{for_sint}, one can compute the weights
$w'_{pq}$ for all $(p,q)\in\Pi$, step-by-step by increasing $p,q$; they
are expressed as
  \begin{eqnarray}
 w'_{m+1,1}&=& g(\pi(m+1,1)); \label{eq:g-Wp} \\
 w'_{p,1}&=&g(\pi(p,1))-g(\pi(p-1,1))
     \qquad \mbox{if $p>m+1$;} \nonumber \\
  w'_{pq} &=&g(\pi(p,q))-g(\pi(p,q-1))\qquad
   \mbox{if $p=\max\{q,m+1\}$ and $q>1$;} \nonumber \\
  w'_{pq} &=&g(\pi(p,q))+g(\pi(p-1,q-1))-g(\pi(p-1,q))
                          -g(\pi(p,q-1)) \nonumber \\
   &\phantom{=}& \qquad\qquad \mbox{otherwise.} \nonumber
  \end{eqnarray}
(This will also be used in Section~\ref{sec:Laurent}.)
Thus,~\refeq{for_int}--\refeq{for_sint} (or~\refeq{g-Wp}) gives a
bijection between $\Wscr'(m,m')$ and $\Rset^{\Bscr^0}$.

We assert that the matrix $W=W'+W''$ is as required in
Proposition~\ref{pr:matr} for the given $g$, i.e., $g(S)=f_w(S)$ holds for
all intervals and sesqui-intervals $S\in \Bscr$.

To show this, first of all observe that the matrix $W''$ in~\refeq{Wpp}
satisfies
   \begin{equation} \label{eq:zerosum}
   w''([p]\times[q])=0 \qquad\mbox{if $p,q\ge m$}.
   \end{equation}

Consider an interval $I=[c..d]\in\Bscr$. In the grid $\Gamma=\Gamma_{n,m'}$
there is a unique admissible flow $\Fscr$ having the source set
$\{s_p\colon p\in[c..d]\}$. This flow consists of the paths
$P_1,\ldots,P_{d-c+1}$, where each $P_i$ begins at the source $s_{\bar i}$
for $\bar i:=c+i-1$, ends at the sink $t_i$, and is of the form
   $$
   P_i=(s_{\bar i}=v_{\bar i,1},v_{\bar i,2},\ldots,v_{\bar i,i},v_{\bar
   i-1,i},\ldots, v_{1,i}=t_i).
   $$
(Hereinafter we use notation for a path without indicating its edges.) The
picture below illustrates $\Fscr$ in the case $c=4$ and $d=6$.

   \begin{center}
  \unitlength=1mm
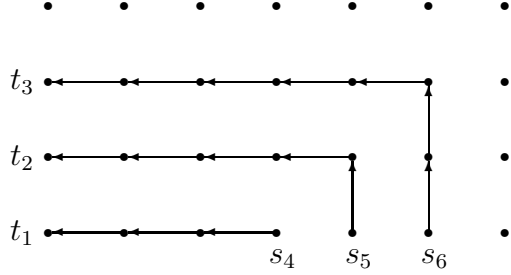
    \begin{picture}(65,40)
 \multiput(10,5)(0,10){4}%
   {\multiput(0,0)(10,0){7}%
     {\put(0,0){\circle*{1}}}
        }
 \multiput(10,5)(10,0){3}%
    {\put(10,0){\vector(-1,0){9.5}}}
 \multiput(10,15)(10,0){4}%
    {\put(10,0){\vector(-1,0){9.5}}}
 \multiput(10,25)(10,0){5}%
    {\put(10,0){\vector(-1,0){9.5}}}
\multiput(50,5)(0,10){1}%
    {\put(0,0){\vector(0,1){9.5}}}
\multiput(60,5)(0,10){2}%
    {\put(0,0){\vector(0,1){9.5}}}
 \put(39,1){$s_4$}
 \put(49,1){$s_5$}
 \put(59,1){$s_6$}
 \put(5,4){$t_1$}
 \put(5,14){$t_2$}
 \put(5,24){$t_3$}
  \end{picture}
 \end{center}

So the function $f$ generated by
$W$ via the flow model satisfies
   $$
 f(I)=w([d]\times[d-c+1])=w'([d]\times[d-c+1])=g(I)
  $$
(in view of~\refeq{zerosum} and $|I|\ge m$).

Next consider a sesqui-interval $S=[1..d_1]\sqcup[c_2..d_2]$ in
$\Bscr$. We associate to a path $P$ from a source $s_i$ to a sink $t_j$ in
$\Gamma$ the closed region of the plane bounded by $P$, by the horizontal
path from $s_i$ to $v_{11}$ and by the vertical path from $v_{11}$ to
$t_j$; we call it the {\em lower region} of $P$ and denote by $\Rscr(P)$.
(Regions of this sort will be used in the next subsection as well.)

In contrast to intervals, the sesqui-interval $S$ admits several
admissible flows with the source set $\{s_p\colon p\in S\}$. We
distinguish one flow $\Fscr$ among these; it is called the {\em lowest}
flow for $S$ and consists of the paths $P_1,\ldots,P_m$ whose lower
regions are as small as possible.

These paths are constructed as follows. For $i=1,\ldots,d_1$, the path
$P_i$ (going from $s_i$ to $t_i$) can be chosen uniquely, namely, $P_i$ is
$(v_{i,1},\ldots,v_{i,i},\ldots,v_{1,i})$. Now let $d_1<i\le m$. Put
$i':=i-d_1$ and $\bar i:=c_2+i'-1$; then $\bar i\in[c_2..d_2]$. One can
see that the path $P_i$, that goes from $s_{\bar i}$ to $t_i$ and has the
minimal lower region (provided that $P_1,\ldots,P_{i-1}$ are already
constructed), is of the form
  $$
  P_i=(v_{\bar i,1},\ldots,v_{\bar i,i'},\ldots,v_{d_1+i',i'},\ldots,
  v_{d_1+i',i},\ldots,v_{1,i})
  $$
(as a rule, $P_i$ makes three turns). An instance of a lowest flow is
illustrated in the picture; here $m=5$, $d_1=2$, $c_2=6$ and $d_2=8$.

   \begin{center}
  \unitlength=1mm
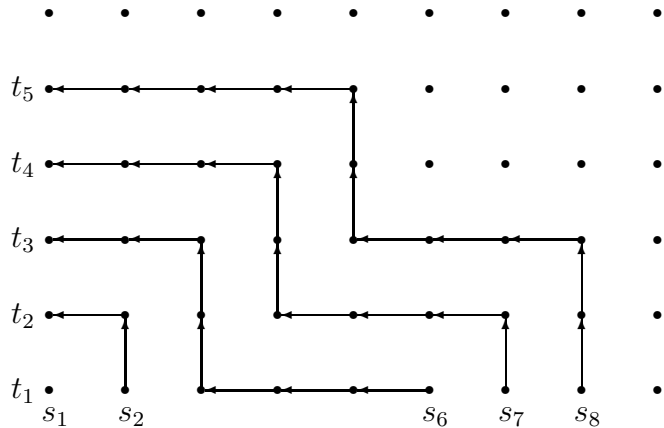
    \begin{picture}(95,60)
 \multiput(10,5)(0,10){6}%
   {\multiput(0,0)(10,0){9}%
     {\put(0,0){\circle*{1}}}
        }
   \put(20,15){\vector(-1,0){9.5}}
   \put(20,5){\vector(0,1){9.5}}
 \multiput(30,5)(10,0){3}%
    {\put(10,0){\vector(-1,0){9.5}}}
 \multiput(40,15)(10,0){3}%
    {\put(10,0){\vector(-1,0){9.5}}}
 \multiput(50,25)(10,0){3}%
    {\put(10,0){\vector(-1,0){9.5}}}
 \multiput(10,35)(10,0){3}%
    {\put(10,0){\vector(-1,0){9.5}}}
 \multiput(10,25)(10,0){2}%
    {\put(10,0){\vector(-1,0){9.5}}}
 \multiput(10,45)(10,0){4}%
    {\put(10,0){\vector(-1,0){9.5}}}
\multiput(30,5)(0,10){2}%
    {\put(0,0){\vector(0,1){9.5}}}
\multiput(40,15)(0,10){2}%
    {\put(0,0){\vector(0,1){9.5}}}
\multiput(50,25)(0,10){2}%
    {\put(0,0){\vector(0,1){9.5}}}
\multiput(70,5)(0,10){1}%
    {\put(0,0){\vector(0,1){9.5}}}
\multiput(80,5)(0,10){2}%
    {\put(0,0){\vector(0,1){9.5}}}
 \put(9,1){$s_1$}
 \put(19,1){$s_2$}
 \put(59,1){$s_6$}
 \put(69,1){$s_7$}
 \put(79,1){$s_8$}
 \put(5,4){$t_1$}
 \put(5,14){$t_2$}
 \put(5,24){$t_3$}
 \put(5,34){$t_4$}
 \put(5,44){$t_5$}
  \end{picture}
 \end{center}

Observe that the vertex set of the lowest flow $\Fscr$ for the given $S$
spans the region being the union of the square $[1..m]\times[1..m]$ and
the rectangle $[m+1..d_2]\times[1..d_2-c_2+1]$. So, in view
of~\refeq{zerosum}, the weight $w(\Fscr)$ of $\Fscr$ amounts to
$w'([m+1..d_2]\times[1..d_2-c_2+1])$, which is equal to
$w'([1..d_2]\times[1..d_2-c_2+1])=g(S)$ (cf.~\refeq{for_sint}).
We assert that $\Fscr$ is the maximum-weight flow for $w$ and $S$.

Indeed, it is not difficult to see that any other admissible flow $\Fscr'$
for $S$ does not cover at least one vertex $v_{pq}$ with $q<p\le m$. This
means that the total contribution to $w(\Fscr')$ from the vertices
$v_{pq}$ such that $1\le p,q\le m$ is at most $-M$, whereas a similar
contribution for $\Fscr$ is zero. Since $M$ is large, we obtain
$w(\Fscr')< w(\Fscr)$ (this is just where the matrix $W''$ is
important). So $f_w(S)=w(\Fscr)=g(S)$, as required.

This completes the proof of Proposition~\ref{pr:matr}.

\subsection{Rearranging flows in the grid} \label{ssec:plan_flows}

Now we start proving Theorem~\ref{tm:flow} claiming that the function $f$ on
$C_m^{m'}$ generated by use of the flow model from any real $n\times m'$
matrix $W$ is indeed a TP-function.
In this subsection we prove a weakened version of this theorem. It is
stated in two lemmas (cf. equalities~\refeq{P3}
and~\refeq{P4}).

  \begin{lemma} \label{lm:geP3}
For an $n\times m'$ matrix $W$ and the function $f=f_w$ on $C_0^{m'}$,
  \begin{equation}  \label{eq:geP3}
   f(Xik)+f(Xj)\ge \max\{f(Xij)+f(Xk),\; f(Xi)+f(Xjk)\}
  \end{equation}
holds for all $i<j<k$ and $X\subset [n]-\{i,j,k\}$ with $|X|\le m'-2$.
 \end{lemma}
  \begin{proof}
We essentially use the fact that the grid $\Gamma=\Gamma_{n,m'}$ is a
planar graph. Let $r:=|X|+2$. W.l.o.g., we may assume that
$f(Xij)+f(Xk)\ge f(Xi)+f(Xjk)$. Let $\Fscr=\{P_1,\ldots,P_r\}$ be an
admissible flow in $\Gamma$ for $Xij$ (i.e., going from the sources
$\{s_p\colon p\in X\}\cup\{s_i,s_j\}$ to the sinks $\{t_1,\ldots,t_r\}$)
such that $f(Xij)=w(\Fscr)$ (cf.~\refeq{f-w}), and
$\Fscr'=\{P'_1,\ldots,P'_{r-1}\}$ an admissible flow for $Xk$ such that
$f(Xk)=w(\Fscr')$.

We combine these flows into one family
$\Pscr=\{P_1,\ldots,P_r,P'_1,\ldots,P'_{r-1}\}$ (possibly containing
repeated paths). Observe that

(i) each vertex of $\Gamma$ belongs to at most two paths in $\Pscr$;

(ii) for $p\in[n]$, the source $s_p$ is the beginning of exactly one path
in $\Pscr$ if $p\in\{i,j,k\}$, and the beginning of exactly two paths if
$p\in X$;

(iii) each of the sinks $t_1,\ldots,t_{r-1}$ is the end of exactly two
paths in $\Pscr$, and $t_t$ is the end of exactly one path.

Using a standard planar flow decomposition technique, one can rearrange
paths in $\Pscr$ so as to obtain a family $\Qscr=\{Q_1,\ldots,Q_{2r-1}\}$
of paths from sources to sinks having properties~(ii),(iii) as above
(with $\Qscr$ in place of $\Pscr$), and in addition:

(iv) for each vertex $v$ of $\Gamma$, the numbers of occurrences of $v$ in
paths of $\Qscr$ and in paths of $\Pscr$ are equal;

(v) $\Rscr(Q_1)\subseteq\Rscr(Q_2)\subseteq\ldots\subseteq
\Rscr(Q_{2r-1})$,

\noindent where $\Rscr(P)$ is the lower region of a path $P$, defined in
Subsection~\ref{ssec:f-matr}. (Such a $\Qscr$ is constructed uniquely.)
Partition $\Qscr$ into two subfamilies:
  \begin{equation} \label{eq:F1F2}
  \Fscr_1:=\{Q_p\colon p\;\;\mbox{is odd}\} \quad\mbox{and}\quad
  \Fscr_2:=\{Q_p\colon p\;\;\mbox{is even}\}
  \end{equation}

We assert that each of these subfamilies consists of pairwise disjoint
paths. Indeed, suppose this is not so. Then, in view of (v), some
subfamily contains ``consecutive'' paths $Q_p,Q_{p+2}$ that share a common
vertex $v$. But now the inclusions $\Rscr(Q_p)\subseteq\Rscr(Q_{p+1})
\subseteq\Rscr(Q_{p+2})$ imply that $v$ must belong to the third path
$Q_{p+1}$ as well, which is impossible by (i) and (iv).

This assertion together with (ii),(iii),(iv) easily implies that both
$\Fscr_1,\Fscr_2$ are admissible flows, that the set of the beginning
vertices of paths in $\Fscr_1$ consists of the sources $s_i$,$s_k$ and
$s_p$ for all $p\in X$, and that the set of the beginning vertices of
paths in $\Fscr_2$ consists of the sources $s_j$ and $s_p$ for all $p\in
X$. Here we use the fact that, due to $i<j<k$, the paths in $\Qscr$
beginning at $s_i,s_j,s_k$ have odd, even and odd indices, respectively.

Thus, $\Fscr_1$ and $\Fscr_2$ are admissible flows for $Xik$ and $Xj$,
respectively. Also (iv) implies
$w(\Fscr_1)+w(\Fscr_2)=w(\Fscr)+w(\Fscr')$. By the definition of $f$, we
have $f(Xik)\ge w(\Fscr_1)$ and $f(Xj)\ge w(\Fscr_2)$, and~\refeq{geP3}
follows.
  \end{proof}

  \begin{lemma} \label{lm:geP4}
For an $n\times m'$ matrix $W$ and the function $f=f_w$ on $C_0^{m'}$,
  \begin{equation}  \label{eq:geP4}
   f(Xik)+f(Xj\ell)\ge \max\{f(Xij)+f(Xk\ell),\; f(Xi\ell)+f(Xjk)\}
  \end{equation}
holds for all $i<j<k<\ell$ and $X\subset [n]-\{i,j,k,\ell\}$ with $|X|\le
m'-2$.
 \end{lemma}
  \begin{proof}
It is similar to the above proof. Assume $f(Xij)+f(Xk\ell)\ge
f(Xi\ell)+f(Xjk)$ (when the reverse inequality holds, the method is
similar). Take in $\Gamma_{n,m'}$ an admissible flow $\Fscr$ for $Xij$
such that $f(Xij)=w(\Fscr)$, and an admissible flow $\Fscr'$ for $Xk\ell$
such that $f(Xk\ell)=w(\Fscr')$. Combine these into one family $\Pscr$,
and rearrange paths in $\Pscr$ so as to obtain a (unique) family
$\Qscr=\{Q_1,\ldots,Q_{2m}\}$ of paths from sources to sinks satisfying
(iv),(v) as in the above proof (with $2m$ in place of $2r-1$). Partition
$\Fscr$ into subfamilies $\Fscr_1,\Fscr_2$ as in~\refeq{F1F2}. Then
$i<j<k<\ell$ implies that $\Fscr_1$ is an admissible flow for $Xik$, and
$\Fscr_2$ is an admissible flow for $Xj\ell$, whence~\refeq{geP4} follows.
  \end{proof}

\subsection{Getting equalities~\refeq{P3}--\refeq{P4}} \label{ssec:final}

It remains to show that the inequalities reverse
to~\refeq{geP3},\refeq{geP4} are valid as well, i.e.,
  \begin{equation} \label{eq:lP3}
     f(Xik)+f(Xj)\le \max\{f(Xij)+f(Xk),\; f(Xi)+f(Xjk)\}
  \end{equation}
and
  \begin{equation} \label{eq:lP4}
   f(Xik)+f(Xj\ell)\le \max\{f(Xij)+f(Xk\ell),\; f(Xi\ell)+f(Xjk)\},
  \end{equation}
hold for a function $f$ derived from an $n\times m'$ matrix $W$ and for
corresponding $X,i,j,k,\ell$. In fact, these relations can be obtained
from a result in~\cite[p.~60]{M} (whereas~\refeq{geP3},\refeq{geP4}
cannot). To make our description self-contained, we give direct proofs
of~\refeq{lP3} and~\refeq{lP4} here, by arguing in a similar spirit as
in~\cite{M}.

To prove~\refeq{lP3}, we take in $\Gamma=\Gamma_{n,m'}$ an admissible flow
$\Fscr$ for $Xik$ with $w(\Fscr)=f(Xik)$, and an admissible flow $\Fscr'$
for $Xj$ with $w(\Fscr')=f(Xj)$. Regarding $\Fscr$ as a graph, we modify
it as follows. Each vertex $v$ of $\Fscr$ is replaced by edge
$e_v=(v',v'')$; each original edge $(u,v)$ of $\Fscr$ is transformed into
edge $(u'',v')$. The resulting graph, consisting of pairwise disjoint
paths as before, is denoted by $\gamma(\Fscr)$. The graph $\Fscr'$ is
modified into $\gamma(\Fscr')$ in a similar way. Corresponding edges of
$\gamma(\Fscr)$ and $\gamma(\Fscr')$ are identified.

Next we construct an auxiliary graph $H$ by the following rule:

(a) if $e$ is an edge in $\gamma(\Fscr)$ but not in $\gamma(\Fscr')$, then
$e$ is included in $H$;

(b) if $e=(u,v)$ is an edge in $\gamma(\Fscr')$ but not in
$\gamma(\Fscr)$, then the edge $(v,u)$ {\em reverse} to $e$ is included in
$H$.

(Common edges of $\gamma(\Fscr),\gamma(\Fscr')$ are not included
in $H$.) One can see that $H$ has the following properties: each vertex
has at most one incoming edge and at most one outgoing edge; the vertices
having one outgoing edge and no incoming edge are exactly $s'_i,s'_k$; the
vertices having one incoming edge and no outgoing edge are exactly
$s'_j,t''_{r}$, where $r=|X|+2$. This implies that $H$ is represented as
the disjoint union of cycles, isolated vertices and two paths $P,Q$, where
either $P$ is a path from $s'_i$ to $s'_j$ and $Q$ is a path from $s'_k$
to $t''_r$ ({\em Case~1}), or $P$ is a path from $s'_k$ to $s'_j$ and $Q$
is a path from $s'_i$ to $t''_r$ ({\em Case~2}).

We use the path $P$ to rearrange the graphs $\gamma(\Fscr)$ and
$\gamma(\Fscr')$ as follows: for each edge $e=(u,v)$ of $P$,

(c) if $e$ is in $\gamma(\Fscr)$, then we delete $e$ from $\gamma(\Fscr)$
and add to $\gamma(\Fscr')$;

(d) if $e$ is not in $\gamma(\Fscr)$, and therefore, the edge $\bar
e=(v,u)$ reverse to $e$ is in $\gamma(\Fscr')$, then we delete $\bar e$
from $\gamma(\Fscr')$ and add to $\gamma(\Fscr)$.

Let $\Gscr$ and $\Gscr'$ be the graphs
obtained in this way from $\gamma(\Fscr)$ and $\gamma(\Fscr')$, respectively
(if there appear isolated vertices, we ignore them). In these
graphs we shrink each edge of the form $e_v=(v',v'')$ into one vertex $v$.
This produces subgraphs $\Fscr_1$ and $\Fscr_2$ of $\Gamma$, where the
former corresponds to $\Gscr$, and the latter to $\Gscr'$.

It is not difficult to deduce from (a)--(d) that each of $\Fscr_1,\Fscr_2$
consists of pairwise disjoint paths, and moreover: in Case~1, $\Fscr_1$ is
an admissible flow for $Xjk$ and $\Fscr_2$ is an admissible flow for $Xi$,
while in Case~2, $\Fscr_1$ is an admissible flow for $Xij$ and $\Fscr_2$
is an admissible flow for $Xk$. Also one can see that for each vertex $v$
of $\Gamma$, the numbers of occurrences of $v$ in paths of
$\{\Fscr_1,\Fscr_2\}$ and in paths of $\{\Fscr,\Fscr'\}$ are the same.
Therefore, $w(\Fscr)+w(\Fscr')=
w(\Fscr_1)+w(\Fscr_2)\le f(Xij)+f(Xk)$, yielding~\refeq{lP3}.

 \medskip
Finally, \refeq{lP4} is shown by a similar method, and we give here only a
short outline, leaving details to the reader. Take in $\Gamma$ an
admissible flow $\Fscr$ for $Xik$ with $w(\Fscr)=f(Xik)$, and an
admissible flow $\Fscr'$ for $Xj\ell$ with $w(\Fscr')=f(Xj\ell)$.
Construct the corresponding graphs $\gamma(\Fscr)$, $\gamma(\Fscr')$ and
$H$. Then $H$ is represented as the disjoint union of cycles, isolated
vertices and two paths $P,Q$, where either $P$ is a path from $s'_i$ to
$s'_j$ and $Q$ is a path from $s'_k$ to $s'_\ell$ ({\em Case~1$'$}), or
$P$ is a path from $s'_k$ to $s'_j$ and $Q$ is a path from $s'_i$ to
$s'_\ell$ ({\em Case~2$'$}). Using the path $P$, we rearrange
$\gamma(\Fscr),\gamma(\Fscr')$ according to (c),(d), eventually obtaining
subgraphs $\Fscr_1,\Fscr_2$ of $\Gamma$. Then both $\Fscr_1,\Fscr_2$ are
admissible flows. Furthermore, in Case~1$'$, $\Fscr_1$ is a flow for $Xjk$
and $\Fscr_2$ is a flow for $Xi\ell$, while in Case~2$'$, $\Fscr_1$ is a
flow for $Xij$ and $\Fscr_2$ is a flow for $Xk\ell$. Also
$w(\Fscr_1)+w(\Fscr_2)=w(\Fscr)+w(\Fscr')$, and~\refeq{lP4} follows.

This gives Theorem~\ref{tm:flow} and completes the proof of
Theorem~A$'$. \qed \qed

 \medskip
\noindent {\bf Remark 1.} As is said in the Introduction, Theorem~A
implies that any TP-function on a truncated box can be extended to a
TP-function on the entire box. In the Boolean case, this property can also
be immediately seen from the flow-in-matrix method. Indeed, given a
TP-function $f$ on the truncated cube $C_m^{m'}$, take an $n\times m'$
matrix $W$ determining $f$ and extend $W$ arbitrarily into an $n\times n$
matrix $W'$. Then $W'$ generates the desired TP-extension of $f$ to the
Boolean cube $2^{[n]}$.


\section{Surjectivity in the general case} \label{sec:gen}

In this section we complete the proof of Theorem~A in the general case, by
showing that the corresponding restriction map $res$ is surjective. We
will use a reduction to the Boolean case and Theorem~A$'$.

\subsection{Surjectivity in the case of an entire box.}
\label{ssec:ent_box}

We start with considering an arbitrary truncated box $B_m^{m'}(a)$. For
$i=1,\ldots,n$, denote $a_1+\ldots+a_i$ by $\bar a_i$, and let $N:=\bar
a_n=|a|$. The ordered set $[N]$ is naturally partitioned into intervals ({\em
blocks}) $L_1,\ldots,L_n$, where $L_i:=\bar a_{i-1}+[a_i]=[\bar
a_{i-1}+1..\bar a_i]$ (letting $\bar a_0:=0$). We associate to a vector
$x\in B_m^{m'}(a)$ the subset $[x]$ of $[N]$ such that
  \begin{numitem}
for $i=1,\ldots,n$, the set $[x]\cap L_i$ consists of $x_i$ {\em
beginning} elements of $L_i$ (i.e., is of the form $\bar a_{i-1}+[x_i]$).
   \label{eq:map[]}
   \end{numitem}
This gives the map $[\,]:B_m^{m'}(a)\to C_m^{m'}(N)$, where $C_m^{m'}(N)$
is the truncated cube formed by the sets $S\subseteq [N]$ with $m\le
|S|\le m'$. Conversely,
  \begin{numitem}
for $S\subseteq [N]$, define $\#(S)$ to be the vector $x\in B(a)$ such
that $x_i=|S\cap L_i|$, $i=1,\ldots,n$.
   \label{eq:mapback}
   \end{numitem}
This gives the map $\#:C_m^{m'}(N)\to B_m^{m'}(a)$. Observe that
$\#([x])=x$. The map $[\,]$ induces the
corresponding map $[\,]^\ast$ of the functions $g$ on $C_m^{m'}(N)$ to
functions $f$ on $B_m^{m'}(a)$: the value of $f$ on $x\in B_m^{m'}(a)$ is
equal to $g([x])$. We observe that
   \begin{numitem}
$[\,]^\ast(g)$ brings a TP-function $g$ on $C_m^{m'}(N)$ to a TP-function
$f$ on $B_m^{m'}(a)$.
    \label{eq:g-f}
    \end{numitem}
Indeed, one can see that for the six vectors $x'$ occurring in
relation~\refeq{GP3}, the corresponding sets $[x']$ are just those that
occur in~\refeq{P3}. So validity of~\refeq{P3} for the latter implies
validity of~\refeq{GP3} for the former. For~\refeq{GP4} and~\refeq{P4},
the argument is similar.

In light of~\refeq{g-f}, the surjectivity of $res$ would follow from the
assertion:
   \begin{numitem}
for any function $f_0:\Bscr\to \Rset$, there is a TP-function $g$ on
$C_m^{m'}(N)$ such that $g([x])=f_0(x)$ holds for all $x\in\Bscr$.
   \label{eq:g-f_0}
   \end{numitem}
(Indeed, the image by $[\,]^\ast$ of $g$ is a TP-function $f$ on
$B_m^{m'}(a)$ coinciding with $f_0$ within $\Bscr$, i.e., having the
required property $res(f)=f_0$.)

In the rest of this subsection we prove~\refeq{g-f_0} in the case when the
box is truncated only from above, i.e., when $m=0$ or, equivalently, $m=1$
(this is technically simpler because in this case $\Bscr$ does not contain
sints). Using this as a base, we will prove, in the next subsection, the
surjectivity of $res$ for any $m$ by applying induction on $m$.

  \begin{prop} \label{pr:surj0}
\refeq{g-f_0} is valid for $\Bscr:=Int(a;1)\cup\ldots\cup I(a;m')$ and
$C_0^{m'}(N)$.
  \end{prop}
  \begin{proof}
By Theorem~A$'$, it is sufficient to assign, in a due way, values of
$g$ on the set of intervals $I$ of $C_0^{m'}(N)$. If $[\#(I)]=I$, i.e.,
if $I$ is the image by $[\,]$ of a fuzzy-interval, the task is trivial: we
put $g(I)=f_0(\#(I))$. But when $[\#(I)]\ne I$, our method of assigning
$g(I)$ via $f_0(\#(I))$ becomes more involved. Moreover, to prove
the correctness of our assignment, we will be forced to express $g$
explicitly for a larger family of sets.

To explain the idea, consider a fint $x$ and let $[c..d]$ be its support.
If $c=d$ or if $x_c=a_c$, then $[x]$ is an interval in $[N]$. In a general
case we partition $[x]$ into two subsets: the {\em head} $H(x):=[x]\cap L_c$
and the {\em tail} $T(x):=[x]\cap[\bar a_c+1..N]$. Clearly both the head
and the tail are intervals (unless $T(x)=\emptyset$).

Besides $[x]$, we produce from $x$ additional sets $Q_1,\ldots,Q_q$ in
$[N]$, where $q=a_c-x_c$ and for $p=1,\ldots,q$, $Q_p$ is obtained from
$[x]$ by shifting its head $H(x)$ by $p$ positions to the right. Formally:
$Q_p:=(p+H(x))\cup T(x)$. The last set $Q_q$ (whose head is pressed to the
end of the block $L_c$) is already a genuine interval. We call each of
$Q_0:=[x],Q_1,\ldots,Q_q$ a {\em quasi-interval}, and associate to each
$Q=Q_p$ three numbers: $h(Q):=x_c$ (the size of the head), $s(Q):=p$ (the
{\em shift} of the head), and $\eps(Q):=s(Q)h(Q)$ (the {\em excess} of
$Q$). Notice that the excess of $[x]$ is zero, whereas the excess of the
interval $Q_q$ is maximal among the quasi-intervals created from $x$.

Let $\Qscr$ be the family of all quasi-intervals in $C_0^{m'}(N)$ (in particular,
$\Qscr$ contains all intervals and the images of all fints). This set is
just where we indicate $g$ explicitly, as follows. Choose a
large positive number $M$ (w.r.t. $f_0$). Define
   \begin{equation} \label{eq:quasi}
   g(Q):=f_0(\#(Q))+M\eps(Q) \qquad\mbox{for each $Q\in\Qscr$}.
   \end{equation}

Then $g$ satisfies the equalities in~\refeq{g-f_0}. It remains to prove
the following

\medskip
 \noindent
{\bf Claim.} {\em Let $g$ be the function on $\Qscr$ defined
by~\refeq{quasi}. Then $g$ is extendable to a TP-function on
$C_0^{m'}(N)$.}

  \medskip
 \noindent
 {\em Proof of the Claim.} Let $g'$ be the TP-function on $C_0^{m'}(N)$
coinciding with $g$ on the set of intervals (the standard basis there). We
show that $g$ and $g'$ are equal on all quasi-intervals as well.

Consider a quasi-interval $Q$ which is not an interval, and let $i=i(Q)$
and $k=k(Q)$ be the first and last elements of $Q$, respectively. We apply
induction on $r(Q):=k(Q)-i(Q)$. Take the element $j\in[N]$ next to the end
of the head of $Q$. Then $j\not\in Q$ and $i<j<k$. Form the sets
$X:=Q-\{i,k\}$ and
  $$
  B:=Xj,\;\; C:=Xij,\;\; D:=Xk,\;\; E:=Xjk,\;\;  F:=Xi.
  $$
These five sets $Q'$ are quasi-intervals (maybe even intervals) with
$r(Q')<r(Q)$. So, by induction, $g'(Q')=g(Q')$ holds for these $Q'$. Let
$h,s,\eps$ stand for $h(Q),s(Q),\eps(Q)$, respectively. A direct
verification shows that
 \begin{eqnarray*}
 &&\eps(Q)=sh;  \\
 &&\eps(B)=(s+1)h\quad\mbox{(the head moves to the right)}; \\
 &&\eps(C)=s(h+1) \quad \mbox{(the head increases at the end)}; \\
 &&\eps(D)=(s+1)(h-1) \quad \mbox{(the head decreases at the beginning)}; \\
 &&\eps(E)=(s+1)h  \quad\mbox{(the head moves to the right)}; \\
 &&\eps(F)=sh \quad\mbox{(the head is stable)}.
  \end{eqnarray*}
(Note that when $h=1$, we have $h(D)\le a_{i+1}$ and $s(D)=0$, so the
expression for $\eps(D)$ gives the correct value 0.) It follows that
  \begin{equation} \label{eq:pairs}
  \eps(Q)+\eps(B)=\eps(E)+\eps(F)>\eps(C)+\eps(D).
   \end{equation}

Since $M$ is large, we obtain from~\refeq{quasi} and~\refeq{pairs} that
$g(E)+g(F)>g(C)+g(D)$. Therefore,
  \begin{equation} \label{eq:equal}
  g'(Q)+g(B)=g(E)+g(F),
  \end{equation}
taking into account that $g'$ is a TP-function and that, by induction,
$g'$ and $g$ are equal on $B,C,D,E,F$. Next, observe that $\#(Q)=\#(E)$
and $\#(B)=\#(F)$. Therefore, $f_0(\#(Q))+f_0(\#(B))=f_0(\#(E))+
f_0(\#(F))$. This together with the equality in~\refeq{pairs} gives
  $$
  g(Q)+g(B)=g(E)+g(F).
   $$
Now comparing this and~\refeq{equal}, we conclude that $g'(Q)=g(Q)$. This
yields the Claim and completes the proof of the proposition. \qed
 \end{proof}

\subsection{Reduction to the entire box.}
\label{ssec:red_box}

To complete the proof of Theorem~A, it remains to show the surjectivity of
the restriction map $res$ for an arbitrary $m$. We apply induction on $m$,
relying on Proposition~\ref{pr:surj0} which gives a base for the
induction.

  \begin{prop}  \label{pr:m-m-1}
Let $0< m\le m'$ and let the restriction map $res':
\Tscr(B_{m-1}^{m'}(a))\to \Rset^{\Bscr'}$ be surjective, where
$\Bscr'=Sint(a;m-1)\cup Int(a;m-1)\cup Int(a;m)\cup\ldots\cup I(a;m')$.
Then $res:\Tscr(B_m^{m'}(a))\to \Rset^{\Bscr}$ is surjective as well.
  \end{prop}
  \begin{proof}
Let $f_0$ be a function on $\Bscr$. Our aim is to construct a function
$g_0$ on $\Bscr'$ satisfying the following conditions:

(a) $g_0$ and $f_0$ are equal on the set $\Dscr:=Int(a;m)\cup\ldots\cup
Int(a;m')$; and

(b) the TP-function $g$ on $B_{m-1}^{m'}(a)$ with $res'(g)=g_0$ satisfies
   \begin{equation}  \label{eq:g-f02}
   g(x)=f_0(x)\qquad \mbox{for each $x\in Sint(a;m)$.}
   \end{equation}
Then (a),(b) imply that the restriction $f$ of $g$ to $B_m^{m'}(a)$ is a
TP-function possessing the desired property $res(f)=f_0$.

The desired function $g_0$ is defined on the vectors in
$\Bscr'-\Dscr=Sint(a;m-1)\cup Int(a:m-1)$ as follows.

For $y\in \Bscr'-\Dscr$, let $p=p(y)$ denote the minimal number such that
$y_p<a_p$. We refer to $p(y)$ as the {\em insertion point} for $y$ and
denote the vector $y+1_p$ by $y^\uparrow$. This $y^\uparrow$ has size $m$
and lies in $B_m^{m'}(a)$. Moreover, $y^\uparrow$ is either a fint or a
sint. Define
     $$
     g_0(y):=f_0(y^\uparrow)+Mt(y),
     $$
where $M$ is a large positive number (w.r.t. $f_0$) and
$t(y):=y_{p+1}+\ldots+y_n$.

We assert that $g_0$ defined this way satisfies~\refeq{g-f02}. To show
this, consider $x\in Sint(a;m)$. Let $\alpha(x),\beta(x)$ be defined as
in~\refeq{alphabeta} (they exist, otherwise $x$ would be a fint), and
assign the parameter $\eta(x)$ as in~\refeq{eta_gen}. Put $i:=\beta(x)$,
$j:=\alpha(x)$ and $k=d(x)$. By the TP3-relation for the function $g$ and
the cortege $(x-1_i-1_k,i,j,k)$, we have
     \begin{equation} \label{eq:g}
               g(x)=\max\{g(C)+g(D),g(E)+g(F)\}-g(B),
     \end{equation}
where $B:=x-1_i +1_j -1_k$, $C:=x+1_j -1_k$, $D:=x-1_i$, $E:=x-1_i+1_j$,
$F:=x-1_k$. We observe the following, letting
$\Sigma:=x_{i+1}+\ldots+x_n$.

 \smallskip
(i) The vectors $C$ and $E$ have size $m$, $C$ is either a fint or a sint
with $\eta(C)<\eta(x)$, and similarly for $E$. So, applying induction on
$\eta$, we have $g(C)=f_0(C)$ and $g(E)=f_0(E)$.

 \smallskip
(ii) The vector $B$ has size $m-1$ and its insertion point is $i$. Then
$B^\uparrow=B+1_i =C$. Also $t(B)=\Sigma+1-1=\Sigma$, whence
$g(B)=f_0(C)+M\Sigma$.

 \smallskip
(iii) The vector $D$ has size $m-1$ and its insertion point is $i$. Then
$D^\uparrow=D+1_i =x$. Also $t(D)=\Sigma$, whence $g(D)=f_0(x)+M\Sigma$.

 \smallskip
(iv) The vector $F$ has size $m-1$ and its insertion point is at least
$i$. This and $F_k=x_k-1$ imply $t(F)\le \Sigma-1$, whence $g(F)\le
f_0(F^\uparrow)+M\Sigma-M$.

\medskip
Since $M$ is large and $t(D)\ge t(F)+M$ (by (iii),(iv)), the maximum
in~\refeq{g} is attained by the first sum occurring there. Therefore, in
view of (i)--(iii),
   $$
   g(x)=g(C)+g(D)-g(B)=f_0(C)+(f_0(x)+M\Sigma)-(f_0(C)+M\Sigma)=f_0(x),
   $$
as required, yielding the proposition.
  \end{proof}

This completes the proof of Theorem~A. \qed\qed


\section{Bases and rhombic tiling diagrams} \label{sec:tiling}

In this section, we first define rhombic tiling diagrams (tilings) and
their dual objects, wiring diagrams. Then we explain a correspondence
between tilings and bases for TP-functions on a box that are derived from
the standard basis by a series of special mutations, so-called normal
bases. Then we characterize the subsets of a box that can be extended to a
normal basis. Finally, we demonstrate some applications of this
characterization, in particular, that a TP-function on a box is extendable
to a larger box.

\subsection{Tilings and wirings} \label{ssec:til-wir}

As before, let $a$ be an $n$-tuple of positive integers. By a {\em rhombic
tiling diagram}, or an {\em RT-diagram}, we mean the following
construction.

In the half-plane $\Rset\times \Rset_+$ take $n$ vectors
$\xi_1,\ldots,\xi_n$ so that: (i) these vectors have equal Euclidean norms
and are ordered clockwise around $(0,0)$, and (ii) all integer
combinations of these vectors are different. Then the set
$Z(a):=\{\lambda_1\xi_1+\ldots+ \lambda_n\xi_n\colon \lambda_i\in\Rset,\;
0\le\lambda_i\le a_i,\; i=1,\ldots,n\}$ is a $2n$-gone. Moreover, $Z=Z(a)$
is a {\em zonogon}, as it is the sum of the segments $[0,a_i\xi_i]$. (Also
it is a linear projection of the solid box ${\rm conv}(B(a))$ into the
plane.) Its left boundary $L$, from the minimal point $p_0:=(0,0)$ to the
maximal point $p_n:=a_1\xi_1+\ldots+a_n\xi_n$, is formed by the points
(vertices) $p_i:=a_1\xi_1+\ldots+a_i\xi_i$ ($i=0,\ldots,n$) connected by
the segments $p_{i-1}p_i:=\{\lambda p_{i-1}+(1-\lambda)p_i\colon
0\le\lambda\le 1\}$. Its right boundary $R$, from $p_0=:p'_n$ to
$p_n=:p'_0$, is formed by the points $p'_i:=a_i\xi_i+\ldots+a_n\xi_n$
($i=0,\ldots,n$) connected by the segments $p'_ip'_{i-1}$.

An {\em RT-diagram} $D$ is a subdivision of the zonogon $Z$ into
``little'' rhombi of the form $x+\{\lambda_i\xi_i+\lambda_j\xi_j\colon
0\le \lambda_i,\lambda_j\le 1\}$ for some $i<j$ and a point $x$ in $Z$.
Such a rhombus is called an $ij$-{\em rhombus}. The diagram $D$ may also
be regarded as a directed planar graph $(V(D),E(D))$ whose vertices and
edges are the vertices and side segments of the rhombi, respectively. Then
each edge $e$ corresponds to a parallel translation of some vector $\xi_i$
and is directed accordingly. Two instances are illustrated in
Fig.~\ref{fig:RT}.

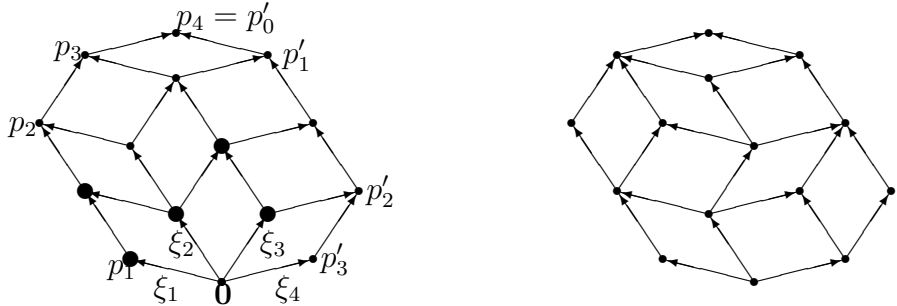
\begin{figure}[htb]
 \begin{center}
  \unitlength=1mm
  \begin{picture}(135,40)(0,5)
  \put(41,5){\vector(-4,1){11.5}}
  \put(29,8){\vector(-2,3){5.7}}
  \put(23,17){\vector(-2,3){5.7}}
  \put(17,26){\vector(2,3){5.7}}
  \put(23,35){\vector(4,1){11.5}}
  \put(47,35){\vector(-4,1){11.5}}
  \put(53,26){\vector(-2,3){5.7}}
  \put(59,17){\vector(-2,3){5.7}}
  \put(53,8){\vector(2,3){5.7}}
  \put(41,5){\vector(4,1){11.5}}
  \put(41,5){\vector(-2,3){5.7}}
  \put(35,14){\vector(-2,3){5.7}}
  \put(47,14){\vector(-2,3){5.7}}
  \put(41,23){\vector(-2,3){5.7}}
  \put(35,14){\vector(-4,1){11.5}}
  \put(29,23){\vector(-4,1){11.5}}
  \put(35,32){\vector(-4,1){11.5}}
  \put(35,14){\vector(2,3){5.7}}
  \put(29,23){\vector(2,3){5.7}}
  \put(41,5){\vector(2,3){5.7}}
  \put(35,32){\vector(4,1){11.5}}
  \put(41,23){\vector(4,1){11.5}}
  \put(47,14){\vector(4,1){11.5}}
  \put(32,3){$\xi_1$}
  \put(34,9){$\xi_2$}
  \put(46,9){$\xi_3$}
  \put(48,3){$\xi_4$}
  \put(40,2){${\bf 0}$}
  \put(26,6){$p_1$}
  \put(13,25){$p_2$}
  \put(19,35){$p_3$}
  \put(35,39){$p_4=p'_0$}
  \put(49,34){$p'_1$}
  \put(60,16){$p'_2$}
  \put(54,7){$p'_3$}
  \put(29,8){\circle*{2}}
  \put(23,17){\circle*{2}}
  \put(35,14){\circle*{2}}
  \put(41,23){\circle*{2}}
  \put(47,14){\circle*{2}}
  \put(41,5){\circle*{1}}
  \put(17,26){\circle*{1}}
  \put(23,35){\circle*{1}}
  \put(47,35){\circle*{1}}
  \put(53,26){\circle*{1}}
  \put(59,17){\circle*{1}}
  \put(53,8){\circle*{1}}
  \put(41,5){\circle*{1}}
  \put(29,23){\circle*{1}}
  \put(35,32){\circle*{1}}
  \put(35,38){\circle*{1}}
%
  \put(111,5){\vector(-4,1){11.5}}
  \put(99,8){\vector(-2,3){5.7}}
  \put(93,17){\vector(-2,3){5.7}}
  \put(87,26){\vector(2,3){5.7}}
  \put(93,35){\vector(4,1){11.5}}
  \put(117,35){\vector(-4,1){11.5}}
  \put(123,26){\vector(-2,3){5.7}}
  \put(129,17){\vector(-2,3){5.7}}
  \put(123,8){\vector(2,3){5.7}}
  \put(111,5){\vector(4,1){11.5}}
  \put(111,5){\vector(-2,3){5.7}}
  \put(99,26){\vector(-2,3){5.7}}
  \put(123,8){\vector(-2,3){5.7}}
  \put(111,23){\vector(-2,3){5.7}}
  \put(105,14){\vector(-4,1){11.5}}
  \put(111,23){\vector(-4,1){11.5}}
  \put(105,32){\vector(-4,1){11.5}}
  \put(105,14){\vector(2,3){5.7}}
  \put(93,17){\vector(2,3){5.7}}
  \put(117,17){\vector(2,3){5.7}}
  \put(105,32){\vector(4,1){11.5}}
  \put(111,23){\vector(4,1){11.5}}
  \put(105,14){\vector(4,1){11.5}}
  \put(111,5){\circle*{1}}
  \put(99,8){\circle*{1}}
  \put(93,17){\circle*{1}}
  \put(87,26){\circle*{1}}
  \put(93,35){\circle*{1}}
  \put(117,35){\circle*{1}}
  \put(123,26){\circle*{1}}
  \put(129,17){\circle*{1}}
  \put(123,8){\circle*{1}}
  \put(99,26){\circle*{1}}
  \put(111,23){\circle*{1}}
  \put(105,32){\circle*{1}}
  \put(105,38){\circle*{1}}
  \put(105,14){\circle*{1}}
  \put(117,17){\circle*{1}}
  \end{picture}
 \end{center}
\caption{Instances of RT-diagrams for $n=4$ and $a=(1,2,1,1)$.}
\label{fig:RT}
  \end{figure}

All diagrams (considered as graphs) are subgraphs of the graph $G$ whose
vertex set $V(G)$ consists of the points
$\pi(x):=x_1\xi_1+\ldots+x_n\xi_n$ for $x\in B(a)$ and whose edge set
$E(G)$ is formed by the pairs $e=(\pi(x),\pi(x'))$ such that
$\pi(x')=\pi(x)+\xi_i$ for some $i\in[n]$; we say that $e$ is an $i$-{\em
edge}. (Due to condition (ii) above, all points $\pi(x)$ are different;
also $\pi(x)$ lies in $Z$.) Observe that $G$ is graded in each color, in
the sense that for all $i\in[n]$ and $u,v\in V(G)$, any paths from $u$ to
$v$ (when exist) have the same number of $i$-edges. The {\em height}
$h(v)$ of a vertex $v$ is the length (number of edges) of a path from the
minimal vertex $p_0$ to $v$.

Let $D$ be an RT-diagram. By a {\em dual path} in $D$ we mean a maximal
sequence $Q=(e_0,\rho_1,e_1,\ldots,\rho_d,e_d)$, where:
$\rho_1,\ldots,\rho_d$ are rhombi of $D$, consecutive rhombi are
different, and $e_{k-1},e_k$ are opposite edges in (the boundary of)
$\rho_k$, $k=1,\ldots,d$. Then all edges in $Q$ are $i$-edges for some
$i\in[n]$, and $Q$ connects the left boundary $L$ and the right boundary
$R$ of the zonogon $Z$. We say that $Q$ is a {\em dual $i$-path} and
assume that $Q$ is oriented so that its first edge $e_0$ belongs to $L$
(and $e_d$ belongs to $R$). So, for each $i\in[n]$, there are exactly
$a_i$ dual $i$-paths and these paths are pairwise disjoint. In particular,
if the first edge $e_0$ of a dual $i$-path $Q$ is a $q$th edge in the side
$p_{i-1}p_i$ of $Z$ (i.e., it connects the point $p_{i-1}+(q-1)\xi_i$ and
$p_{i-1}+q\xi_i$ for some $q\in\{1,\ldots,a_i\}$), then the last edge
$e_d$ of $Q$ is the $q$th edge in the side $p'_{i}p'_{i-1}$; we denote
such edges in $L$ and $R$ by $\ell_{i,q}$ and $r_{i,q}$, respectively.
Also one can see that a dual $i$-path intersects each dual $j$-path with
$j\ne i$ at exactly one rhombus.

In light of these observations and using planar duality, the RT-diagrams
can be associated with so-called wiring diagrams, or W-diagrams (giving
the de Bruijn dualization~\cite{dB}; we use the shorter name ``wire''
rather than ``de Bruijn line''). A {\em W-diagram} is represented by $|a|$
curves, or {\em wires}, $\sigma_{i,q}$ for $i=1,\ldots,n$ and
$q=1,\ldots,a_i$, such that
 \begin{numitem}
each $\sigma_{i,q}$ is identified with (the image of) a continuous
injective map $\sigma$ of the segment $[0,1]$ into $Z$ such that
$\sigma(0)=p_{i,q}$, $\sigma(1)=p'_{i,q}$, and $\sigma(t)$ lies in the
interior of $Z$ for $0<t<1$, where $p_{i,q}$ and $p'_{i,q}$ are the
mid-points of the boundary edges $\ell_{i,q}$ and $r_{i,q}$, respectively,
and the following conditions hold:
   \begin{itemize}
 \item[(a)] any two wires $\sigma_{i,q}$ and $\sigma_{i,q'}$ ($q\ne q'$)
are disjoint, i.e., there are no $0\le t,t'\le 1$ such that
$\sigma_{i,q}(t)=\sigma_{i,q'}(t')$;
 \item[(b)] any two wires $\sigma_{i,q}$ and $\sigma_{j,q'}$ with
$i\ne j$ have exactly one point in common, i.e., there are unique $t,t'$
such that $\sigma_{i,q}(t)=\sigma_{j,q'}(t')$ (these wires {\em cross},
not touch, at this point);
 \item[(c)] no three wires have a common point.
  \end{itemize}
  \label{eq:wires}
  \end{numitem}
Such a diagram is considered up to a homeomorphism of the zonogon $Z$
stable on its boundary, and when needed, we may assume that each wire is
piece-wise linear. Also we orient each wire $\sigma_{i,q}$ from $p_{i,q}$
to $p'_{i,q}$ and call it an $i$-{\em wire}.

There is a natural one-to-one correspondence between the RT- and
W-diagrams. More precisely, for an RT-diagram $D$ and a dual path
$Q=(\ell_{i,q}=e_0,\rho_1,e_1,\ldots,\rho_d,e_d=r_{i,q})$ in it, take as
wire $\sigma_{i,q}$ the concatenation of the segments connecting the
mid-points of edges $e_k,e_{k+1}$; then these wires form a W-diagram. (We
will call such a $\sigma_{i,q}$ the {\em median line} of $Q$.) Conversely,
given a W-diagram $W$, consider the common point $v$ of wires
$\sigma_{i,q}$ and $\sigma_{j,q'}$ with $i<j$ in $W$. For $k=1,\ldots,n$,
let $x_k$ be the number of wires $\sigma=\sigma_{k,q''}$ that {\em go
below} $v$, i.e., $v$ and $p_n$ occur in the same connected region when
the curve $\sigma$ is removed from $Z$. Then we associate to $v$ the
$ij$-rhombus with the minimal vertex at the point
$x_1\xi_1+\ldots+x_n\xi_n$. One can check that conditions (a)--(c)
in~\refeq{wires} provide that the set of rhombi obtained this way forms an
RT-diagram $D$ and, furthermore, that the W-diagram constructed from $D$
as explained above is equivalent to $W$.

Depending on the context, a W-diagram $W$ may also be regarded as a
directed planar graph $(V(W),E(W))$, which is a sort of dual graph of the
corresponding RT-diagram $D$. More precisely, $V(W)$ consists of the
intersection points of pairs of wires (corresponding to the rhombi of $D$)
plus the points $p_{i,q}$ and $p'_{i,q}$ for all $i,q$, and $E(W)$
consists of the wire parts between such vertices, with the orientation
inherited from that of the wires. An edge contained in an $i$-wire is
called an $i$-{\em edge}. Observe that:
  \begin{numitem}
for the intersection point $v$ of an $i$-wire and a $j$-wire with $i<j$,
the edges $e_i,e_j,e'_i,e'_j$ incident with $v$ follow in this order
clockwise around $v$, where $e_i,e'_i$ are the $i$-edges entering and
leaving $v$, respectively, and $e_j,e'_j$ are the $j$-edges entering and
leaving $v$, respectively.
  \label{eq:clockwise}
  \end{numitem}

The assertions in the rest of this section will be stated in terms of
RT-diagrams. In its turn, the language of W-diagrams will mainly be used
to simplify visualization and technical details in the proofs. In
particular, we will sometimes choose triples of wires in a W-diagram, an
$i$-wire $\sigma$, a $j$-wire $\sigma'$ and a $k$-wire $\sigma''$ for
$i<j<k$, and consider the region (curvilinear {\em triangle})
$T=T(\sigma,\sigma',\sigma'')$ in $Z$ bounded by the parts of these wires
between their intersection points. We will distinguish between two cases,
by saying that $\sigma,\sigma',\sigma''$ form the $\Delta$-{\em
configuration} if the point $\sigma\cap\sigma''$ (as well as the triangle
$T$) lies above $\sigma'$, and the $\nabla$-{\em configuration} if this
point lies below $\sigma'$; see Fig.~\ref{fig:config}.
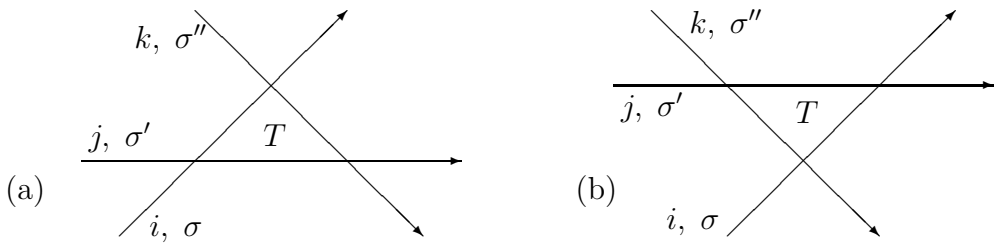
\begin{figure}[htb]
   \begin{center}
  \unitlength=1mm
  \begin{picture}(120,25)(5,5)
  \put(5,10){\vector(1,0){50}}
  \put(10,0){\vector(1,1){30}}
  \put(20,30){\vector(1,-1){30}}
 \put(14,0){$i,\; \sigma$}
 \put(6,12){$j,\; \sigma'$}
 \put(12,25){$k,\; \sigma''$}
 \put(29,12){$T$}
 \put(-5,5){(a)}
  \put(75,20){\vector(1,0){50}}
  \put(90,0){\vector(1,1){30}}
  \put(80,30){\vector(1,-1){30}}
 \put(82,1){$i,\; \sigma$}
 \put(76,16){$j,\; \sigma'$}
 \put(85,27){$k,\; \sigma''$}
 \put(99,15){$T$}
 \put(70,5){(b)}
  \end{picture}
 \end{center}
\caption{(a) $\Delta$-configuration;\;\; (b) $\nabla$-configuration.}
\label{fig:config}
  \end{figure}

\subsection{RT-diagrams and bases} \label{ssec:RT-bases}

In this subsection we discuss relationships between RT-diagrams and
TP-bases in a box $B(a)$. For $S\subseteq B(a)$, let $G_S$ denote the
induced subgraph of $G$ with the vertex set $\pi(S)$ (the graph $G$ is
defined in the previous subsection). We are interested in those subsets
$S$ for which $G_S$ is an RT-diagram, in which case we say that $S$ is an
{\em RT-set}.

An important instance of RT-sets is the set $Int(a)$ of fuzzy-intervals
(including the zero vector), i.e., the standard basis for $B(a)$; see the
left part in Fig.~\ref{fig:RT} where the RT-diagram for $Int(1,2,1,1)$ is
drawn. (This fact can be shown by induction on $n$. Instruction: in case
$n=2$, $G_{Int(a)}$ is the direct product $L_{a_1}\times L_{a_2}$, where
$L_q$ is the path of length $q$; so it is an RT-diagram. In case $n>2$,
let $a'=(a_1,\ldots,a_{n-1})$ and assume by induction that
$G'=G_{Int(a')}$ is an RT-diagram. Let $R'$ be the right boundary of $G'$
(going from the minimal vertex $(0,0)$ to the maximal vertex
$a_1\xi_1+\ldots+a_{n-1}\xi_{n-1}$). Then $G_{Int(a)}$ is obtained by
taking the corresponding union of $G'$ and $R'\times L_{a_n}$.)

Let $\Mscr(a)$ be the set of bases which can be obtained from $Int(a)$ by
a series of TP3-mutations. At the first glance, it seems likely that each
member of $\Mscr(a)$ is an RT-set. However, this is not so. Indeed, in the
standard basis for $B(1,2,1,1)$, take the elements $1_1,1_1+1_2,1_2,
1_3,1_2+1_3$; their images are indicated by thick dots on the left picture
in Fig.~\ref{fig:RT}. These elements give rise to the TP3-mutation
$1_2\rightsquigarrow 1_1+1_3$ of $Int(a)$, but the resulting basis is
already not an RT-set. (By withdrawing the fourth coordinate, we obtain a
similar situation in the 3-dimensional box $B(1,2,1)$.) Another example,
with the Boolean cube $2^{[4]}$, is drawn in the picture below (here the
RT-diagram determines a basis $\Bscr$ but the mutation involving the sets
corresponding to the thick dots, namely, $\{3\}\rightsquigarrow\{1,4\}$,
results in a basis which is not an RT-set).

 \begin{center}
  \unitlength=1mm
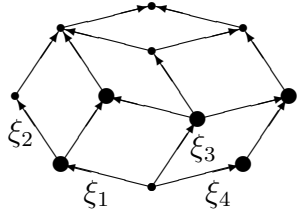
  \begin{picture}(70,25)
  \put(41,0){\vector(-4,1){11.5}}
  \put(29,3){\vector(-2,3){5.7}}
  \put(23,12){\vector(2,3){5.7}}
  \put(29,21){\vector(4,1){11.5}}
  \put(53,21){\vector(-4,1){11.5}}
  \put(59,12){\vector(-2,3){5.7}}
  \put(53,3){\vector(2,3){5.7}}
  \put(41,0){\vector(4,1){11.5}}
  \put(35,12){\vector(-2,3){5.7}}
  \put(47,9){\vector(-2,3){5.7}}
  \put(47,9){\vector(-4,1){11.5}}
  \put(41,18){\vector(-4,1){11.5}}
  \put(29,3){\vector(2,3){5.7}}
  \put(41,0){\vector(2,3){5.7}}
  \put(41,18){\vector(4,1){11.5}}
  \put(47,9){\vector(4,1){11.5}}
  \put(32,-2){$\xi_1$}
  \put(22,6){$\xi_2$}
  \put(46,4){$\xi_3$}
  \put(48,-2){$\xi_4$}
\put(29,3){\circle*{2}}
  \put(35,12){\circle*{2}}
  \put(47,9){\circle*{2}}
  \put(53,3){\circle*{2}}
  \put(59,12){\circle*{2}}
  \put(41,0){\circle*{1}}
  \put(23,12){\circle*{1}}
  \put(41,18){\circle*{1}}
  \put(29,21){\circle*{1}}
  \put(53,21){\circle*{1}}
  \put(53,3){\circle*{1}}
  \put(41,0){\circle*{1}}
  \put(41,24){\circle*{1}}
    \end{picture}
   \end{center}

So, to get bases corresponding to RT-diagrams, we have to restrict the
class of mutations that we apply. Consider a basis $\Bscr$ and a cortege
$(x,i,j,k)$ such that the vectors involving in~\refeq{GP3}, except for one
vector $y\in\{x+1_j,x+1_i+1_k\}$, belong to $\Bscr$. We say that the
mutation $y\rightsquigarrow y'$ (where $\{y,y'\}=\{x+1_j,x+1_i+1_k\}$) is
{\em normal} if both vectors $x$ and $x+1_i+1_j+1_k$ belong to $\Bscr$ as
well. In this case, for $s:=\pi(x)$, the six points
   \begin{equation}  \label{eq:hex}
s,\;\; u:=s+\xi_i,\;\; v:=s+\xi_i+\xi_j,\;\; w:=s+\xi_k,\;\;
z:=s+\xi_j+\xi_k,\;\; t:=s+\xi_i+\xi_j+\xi_k,
    \end{equation}
along with the six edges connecting them, form a (little) {\em hexagon} in
$G_\Bscr$, denoted by $H(s,i,j,k)$. This hexagon $H$ is partitioned into
three rhombi in $D$, either by use of edges $sy,yv,yz$ or by use of edges
$uy,wy,yt$, where $y$ is the unique vertex of $D$ in the interior of $H$;
we refer to $H$ as a $\curlyvee$-{\em hexagon} in the former case, and a
$\curlywedge$-{\em hexagon} in the latter case, see
Fig.~\ref{fig:hexagon}.

\begin{figure}[htb]
 \begin{center}
  \unitlength=1mm
  \begin{picture}(90,25)(0,0)
  \put(10,5){\circle*{1}}
  \put(10,17){\circle*{1}}
  \put(20,0){\circle*{1}}
  \put(20,12){\circle*{1}}
  \put(20,22){\circle*{1}}
  \put(30,5){\circle*{1}}
  \put(30,17){\circle*{1}}
  \put(10,5){\vector(0,1){11.5}}
  \put(20,0){\vector(0,1){11.5}}
  \put(30,5){\vector(0,1){11.5}}
  \put(20,0){\vector(-2,1){9.7}}
  \put(20,12){\vector(-2,1){9.7}}
  \put(30,17){\vector(-2,1){9.7}}
  \put(10,17){\vector(2,1){9.7}}
  \put(20,0){\vector(2,1){9.7}}
  \put(20,12){\vector(2,1){9.7}}
  \put(21,-2){$s$}
  \put(7,3){$u$}
  \put(7,17){$v$}
  \put(21,23){$t$}
  \put(31,3){$w$}
  \put(31,18){$z$}
  \put(19,14){$y$}
  \put(0,0){(a)}
  \put(70,5){\circle*{1}}
  \put(70,17){\circle*{1}}
  \put(80,0){\circle*{1}}
  \put(80,10){\circle*{1}}
  \put(80,22){\circle*{1}}
  \put(90,5){\circle*{1}}
  \put(90,17){\circle*{1}}
  \put(70,5){\vector(0,1){11.5}}
  \put(80,10){\vector(0,1){11.5}}
  \put(90,5){\vector(0,1){11.5}}
  \put(80,0){\vector(-2,1){9.7}}
  \put(90,5){\vector(-2,1){9.7}}
  \put(90,17){\vector(-2,1){9.7}}
  \put(70,17){\vector(2,1){9.7}}
  \put(80,0){\vector(2,1){9.7}}
  \put(70,5){\vector(2,1){9.7}}
  \put(60,0){(b)}
  \end{picture}
 \end{center}
\caption{(a) $\curlyvee$-hexagon;\;\; (b) $\curlywedge$-hexagon.}
 \label{fig:hexagon}
  \end{figure}

It is useful to describe the difference between $\curlyvee$- and
$\curlywedge$-hexagons of $D$ in terms of the corresponding wires in the
W-diagram $W$ related to $D$. More precisely, given a hexagon
$H=H(s,i,j,k)$, let $\sigma,\sigma',\sigma''$ be the wires in $W$
corresponding to the dual $i$-, $j$- and $k$-paths, respectively, that
meet rhombi in $H$. It is easy to see that these wires form the
$\Delta$-configuration if $H$ is a $\curlyvee$-hexagon, and the
$\nabla$-configuration if $H$ is a $\curlywedge$-hexagon (cf.
Figures~\ref{fig:config} and~\ref{fig:hexagon}). Moreover, in both cases
the triangle $T(\sigma,\sigma',\sigma'')$ is {\em inseparable}, which
means that none of the other wires in $W$ goes across this triangle. A
converse property also takes place: if the triangle for three wires
$\sigma,\sigma',\sigma''$ (concerning different $i,j,k\in[n]$) is
inseparable, then the three rhombi in $D$ corresponding to the points
where these wires intersect are assembled in a hexagon.

Let us call a basis for $B(a)$ {\em normal} if it can be obtained from the
standard basis $Int(a)$ by a series of normal TP3-mutations. We show the
following
  \begin{prop} \label{pr:diag}
Each normal basis induces an RT-diagram, and vice versa.
  \end{prop}
  \begin{proof}
Let $\Bscr$ be a basis such that $D:=G_\Bscr$ is an RT-diagram. Then a
normal mutation in $\Bscr$ corresponds to a transformation within some
hexagon $H$ of $D$. This transformation changes the partition of $H$ into
three rhombi by the other partition of $H$ into three rhombi, and we again
obtain an RT-diagram. This implies that each normal basis is an RT-set
(since the standard basis is such).

Next we prove the other direction in the proposition: each RT-diagram $D$
corresponds to a normal basis.

To show this, we use induction on the parameter $h(D):=\sum_{v\in V(D)}
h(v)$, the (total) {\em height of $D$}. Suppose $D$ contains a
$\curlywedge$-hexagon $H$. Then the transformation of $H$ into a
$\curlyvee$-hexagon (which matches a normal mutation) results in an
RT-diagram $D'$ such that $h(D')=h(D)-1$ (as the distance from the minimal
vertex of $H$ to the vertex of the diagram occurring in the interior of
$H$ changes from 2 to 1). By induction, assuming that $D'$ corresponds to
a normal basis, a similar property takes place for $D'$.

To complete the induction, it remains to consider the situation when no
$\curlywedge$-hexagon exists. We assert that
   \begin{numitem}
if $D$ contains no $\curlywedge$-hexagon, then $D$ corresponds to the
standard basis.
  \label{eq:minD}
   \end{numitem}

To show this, consider the W-diagram $W$ related to $D$ and take the wire
$\tau$ in $W$ that goes from the point $p_{n,a_n}$ to $p'_{n,a_n}$. Let
$\Omega$ be the region in $Z$ between $\tau$ and the right boundary $R$.
Two cases are possible.

  \medskip
{\em Case 1}\,: No two wires in $W$ intersect within the interior of
$\Omega$. This means that the dual $n$-path $Q$ in $D$ beginning at the
last edge $\ell_{n,a_n}$ of $L$ and ending at the edge $r_{n,a_n}$ {\em
follow the right boundary} $R$ of $Z$. Then we can reduce the diagram $D$
by removing $R$ and (the interiors of the elements of) $Q$. This gives a
correct RT-diagram for the tuple $a'=(a_1,\ldots,a_{n-1},a_n-1)$,
and~\refeq{minD} follows by applying induction on $|a|$. (When $a_n=1$,
the $n$th entry of $a'$ vanishes.)

  \medskip
{\em Case 2}\,: There are two wires in $W$ that intersect in the interior
of $\Omega$, say, an $i'$-wire $\mu$ and a $j'$-wire $\nu$. Then $i',j',n$
are different. Let for definiteness $i'<j'$. Then the end point of $\mu$
occurs in $R$ later than the end point of $\nu$. This implies that the
point $\mu\cup\tau$ lies below $\nu$, i.e., $\mu,\nu,\tau$ form the
$\nabla$-configuration.

So $W$ contains three wires $\sigma,\sigma',\sigma''$ forming the
$\nabla$-configuration. Let $\eta(\sigma,\sigma',\sigma'')$ denote the
number of wires $\bar\sigma\in W$ such that (the interior of) the triangle
$T=T(\sigma,\sigma',\sigma'')$ lies entirely below $\bar\sigma$. Choose
such wires $\sigma,\sigma',\sigma''$ so that
$\eta(\sigma,\sigma',\sigma'')$ is maximum and, subject to this, the area
of their triangle $T$ is as small as possible. Let for definiteness
$\sigma,\sigma',\sigma''$ be, respectively, $i$-, $j$- and $k$-wires with
$i<j<k$, and let $I,J,K$ denote the $i$-, $j$- and $k$-sides of $T$. (In
view of~\refeq{clockwise}, $J,K$ are directed from the point
$\sigma'\cap\sigma''$, and $I$ is directed from $\sigma\cap\sigma''$.)

 \medskip
 \noindent
 {\bf Claim.} {\em The triangle $T$ is inseparable.}

 \medskip
 \noindent
 {\em Proof of the Claim.}
Suppose there exists a $p$-wire $\hat\sigma$ going across $T$. Then
$\hat\sigma$ crosses two sides among $I,J,K$. We consider five possible
cases and use property~\refeq{clockwise}.

 \smallskip
(a) $\hat\sigma$ first crosses $K$ and then crosses $J$. Then $p<j<k$.
Form the triple $\Sigma:=\{\hat\sigma,\sigma',\sigma''\}$.

 \smallskip
(b) $\hat\sigma$ first crosses $K$ and then crosses $I$. Then $i<p<k$.
Form $\Sigma:=\{\sigma,\hat\sigma,\sigma''\}$.

 \smallskip
(c) $\hat\sigma$ first crosses $J$ and then crosses $I$. Then $i<j<p$.
Form $\Sigma:=\{\sigma,\sigma',\hat\sigma\}$.

 \smallskip
(d) $\hat\sigma$ first crosses $I$ and then crosses $J$. Then $p<i<k$.
Form $\Sigma:=\{\hat\sigma,\sigma,\sigma''\}$.

 \smallskip
(e) $\hat\sigma$ first crosses $J$ and then crosses $K$. Then $i<k<p$.
Form $\Sigma:=\{\sigma,\sigma'',\hat\sigma\}$.

 \smallskip
(The case when $\hat\sigma$ first crosses $I$ and then crosses $K$ is
impossible, otherwise we would have $p<i$ and $k<p$,
by~\refeq{clockwise}.) Let $T'$ denote the triangle of the triple
$\Sigma$. Cases (a)--(e) are illustrated in the picture.
   \begin{center}
  \unitlength=1mm
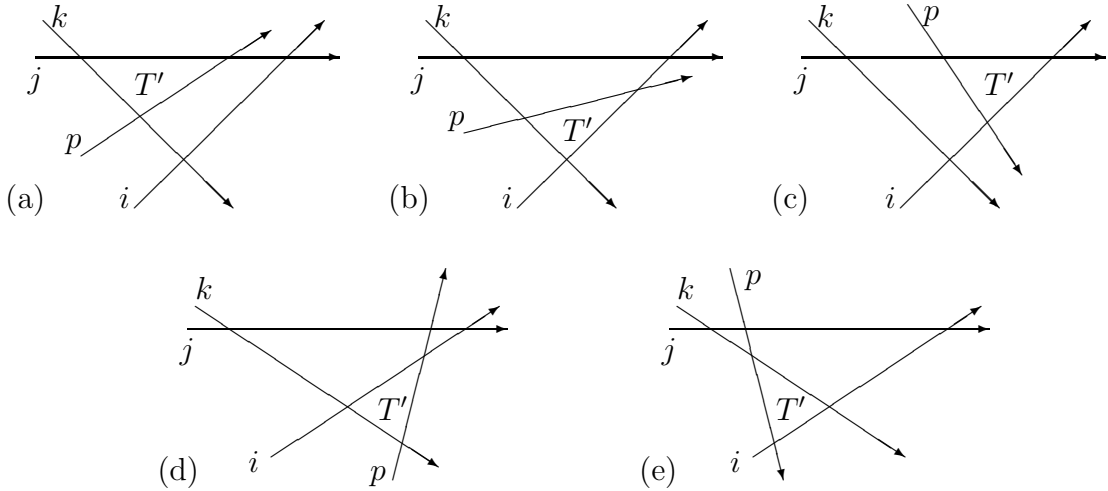
  \begin{picture}(150,30)(0,0)

 \begin{picture}(42,28)(0,0)
  \put(4,20){\vector(1,0){40}}
  \put(17,0){\vector(1,1){25}}
  \put(5,25){\vector(1,-1){25}}
 \put(15,0){$i$}
 \put(3,16){$j$}
 \put(6,24){$k$}
  \put(10,7){\vector(3,2){25}}
 \put(8,8){$p$}
 \put(17,15){$T'$}
 \put(0,0){(a)}
 \end{picture}
 \begin{picture}(42,28)(-7,0)
  \put(4,20){\vector(1,0){40}}
  \put(17,0){\vector(1,1){25}}
  \put(5,25){\vector(1,-1){25}}
 \put(15,0){$i$}
 \put(3,16){$j$}
 \put(6,24){$k$}
  \put(10,10){\vector(4,1){30}}
 \put(8,11){$p$}
 \put(23,9){$T'$}
 \put(0,0){(b)}
 \end{picture}
 \begin{picture}(42,28)(-14,0)
  \put(4,20){\vector(1,0){40}}
  \put(17,0){\vector(1,1){25}}
  \put(5,25){\vector(1,-1){25}}
 \put(15,0){$i$}
 \put(3,16){$j$}
 \put(6,24){$k$}
  \put(18,27){\vector(2,-3){15}}
 \put(20,25){$p$}
 \put(28,15){$T'$}
 \put(0,0){(c)}
 \end{picture}

  \end{picture}
 \end{center}

   \begin{center}
  \unitlength=1mm
  \begin{picture}(110,30)(0,0)
 \begin{picture}(42,28)(0,0)
  \put(4,20){\vector(1,0){42}}
  \put(15,3){\vector(3,2){30}}
  \put(5,23){\vector(3,-2){32}}
 \put(12,1){$i$}
 \put(3,16){$j$}
 \put(5,24){$k$}
  \put(31,0){\vector(1,4){7}}
 \put(28,0){$p$}
 \put(29,8){$T'$}
 \put(0,0){(d)}
 \end{picture}
 \begin{picture}(42,28)(-20,0)
  \put(4,20){\vector(1,0){42}}
  \put(15,3){\vector(3,2){30}}
  \put(5,23){\vector(3,-2){30}}
 \put(12,1){$i$}
 \put(3,16){$j$}
 \put(5,24){$k$}
  \put(12,28){\vector(1,-4){7}}
 \put(14,26){$p$}
 \put(18,8){$T'$}
 \put(0,0){(e)}
 \end{picture}
  \end{picture}
 \end{center}

In all cases, $\Sigma$ forms the $\nabla$-configuration. In cases
(a)--(c), the triangle $T'$ lies inside $T$, which implies that
$\eta(\Sigma)\ge \eta(\sigma,\sigma',\sigma'')$ and that the area of $T'$
is smaller than that of $T$. In cases (d),(e), it is easy to see that when
$T$ lies below some wire, then so does $T'$. Also, in case (d) (resp.
(e)), $T'$ lies below $\sigma$ (resp. $\sigma''$), whereas $T$ does not.
Therefore, in these two cases, $\eta(\Sigma)>
\eta(\sigma,\sigma',\sigma'')$. So we come to a contradiction with the
choice of $\sigma,\sigma',\sigma''$, yielding the Claim. \qed

  \medskip
By the Claim, the three rhombi corresponding to the vertices of $T$ form a
$\curlywedge$-hexagon in $D$. This contradiction completes the proof
of~\refeq{minD} and of the proposition.
  \end{proof}

 \noindent
{\bf Remark 3.} As is seen from the above proof, the diagram $D$
corresponding to the standard basis $Int(a)$ is {\em minimal} in the sense
that the height $h(D)$ of vertices of $D$ is minimum among all
RT-diagrams. Applying the central symmetry to $D$ and reversing all the
edges, we obtain the RT-diagram $D^\ast$ having the maximum height. The
basis corresponding to $D^\ast$ is formed by the ``complementary'' vectors
$a-x$ of the fints $x$ ({\em co-fints}). By Proposition~\ref{pr:diag},
this ``complementary basis'' $Int^\ast(a)$ is normal as well. For a
similar reason, any normal basis $\Bscr$ and its complementary basis
$\Bscr^\ast:=\{x: a-x\in\Bscr\}$ are connected by a series of normal
mutations.

  \medskip
It is not clear to us whether, for some tuple $a$, there exists a TP-basis
beyond $\Mscr(a)$.

\subsection{Normal bases including a given set.} \label{ssec:til-set}

An interesting open question is to characterize the subsets of $B(a)$ that
can be extended to a TP-basis. In this subsection we consider a different
but somewhat related problem:
  \begin{numitem}
Given a set $X\subset B(a)$, decide whether there exists a normal basis
for $B(a)$ including $X$, or, equivalently, an RT-diagram $D$ whose vertex
set contains all points of $\pi(X)$ (the image of $X$ in the zonogon
$Z(a)$).
  \label{eq:problem}
  \end{numitem}

As a variant of such a problem (in fact, equivalent to~\refeq{problem}),
one is given a set of (not necessarily disjoint) subzonogons in $Z(a)$,
with possible rhombic tilings on some of them, and is asked of their
extendability to a tiling on $Z(a)$.

It is useful to reformulate~\refeq{problem} in terms of wiring diagrams.
Then each point $x=(x_1,\ldots,x_n)\in X$ imposes a requirement on the set
of wires {\em going below} (the image of) $x$ in a W-diagram
$W=\{\sigma_{i,q}\}$, assuming that the desired RT-diagram does exist and
regarding the wires in $W$ as being realized by the median lines of the
dual paths. More precisely, if~\refeq{problem} has a solution, then
  \begin{numitem}
for each $i=1,\ldots,n$, the wires $\sigma_{i,1},\ldots,\sigma_{i,x_i}$
should go below $x$, while the other $i$-wires above $x$.
  \label{eq:below-above}
  \end{numitem}

Consider three wires $\sigma=\sigma_{i,q}$, $\sigma'=\sigma_{j,q'}$ and
$\sigma''=\sigma_{k,q''}$ with $i<j<k$. Suppose there is a point $x\in X$
such that $x_i<q$, $x_j\ge q'$ and $x_k<q''$. Then~\refeq{below-above}
forces $\sigma,\sigma',\sigma''$ to form the $\Delta$-configuration.
Another possible situation is when $x_i\ge q$, $x_j< q'$ and $x_k\ge q''$.
In this case, by similar reasons, the point $x$ prescribes the
$\nabla$-configuration for $\sigma,\sigma',\sigma''$.

These observations lead us to the following conclusion:
problem~\refeq{problem} has no solution if
  \begin{numitem}
there are $1\le i<j<k\le n$ and two points $x,x'\in X$ such that
$x_i<x'_i$, $x_j>x'_j$ and $x_k<x'_k$.
  \label{eq:obstacle}
  \end{numitem}
The simplest example is given by the points $2$ and $13$ of the Boolean
cube $2^{[3]}$, which cannot simultaneously occur in one and the same
RT-diagram.

It turns out that~\refeq{obstacle} fully describes obstacles to the
solvability of the problem.

  \begin{theorem} \label{tm:X-tiling}
Problem~\refeq{problem} has solution if and only if~\refeq{obstacle} does
not take place.
  \end{theorem}
  \begin{proof}
We have to show the solvability of~\refeq{problem} if no $i,j,k,x,x'$ as
in~\refeq{obstacle} exist (the other direction in the theorem has been
explained). Our proof is constructive and provides a polynomial time
algorithm of finding a required tiling.

Let $k\in[n]$ and $q\in[a_k]$ and suppose that the wires $\sigma_{i,p}$
are already constructed for all $(i,p)$ such that either $i<k$, or $i=k$
and $p<q$. This means that the W-diagram $W'$ formed by these wires is
realized by use of a directed planar graph $H=(V(H),E(H))$ embedded in $Z$
(according to~\refeq{wires}) and that the requirements as
in~\refeq{below-above} are satisfied; we add to $H$ the boundary of $Z$ as
well. So $V(H)$ consists of the intersection points of wires plus the
points $p_{j,r}$ and $p'_{j,r}$ for all $j\in[n]$ and $r\in[a_j]$. Let
$F(H)$ be the set of inner faces of $H$. We also (conditionally) place
each point $x\in X$ in the interior of some face $f\in F(H)$, which is
chosen according to~\refeq{below-above}, i.e., for each
$\sigma=\sigma_{i,p}\in W'$, $x$ occurs below $\sigma$ if $x_i\le p$, and
above $\sigma$ otherwise (so $f$ is chosen uniquely for $x$).

The faces in $F(H)$ containing the minimal vertex $p_0$ and the maximal
vertex $p_n$ of $Z$ are denoted by $f^{\min}$ and $f^{\max}$,
respectively. It follows from~\refeq{clockwise} that $H$ is acyclic. When
an edge $e$ is an $i$-edge (a part of an $i$-wire), we also say that $e$
has {\em index} $i$.

Our aim is to add wire $\sigma=\sigma_{k,q}$ to $W'$ by transforming $H$
in a due way. Let $\Xdown$ be the set of points of $x\in X$ with $x_k\le
q$, and $\Xup:=X-\Xdown$ (the sets of points that should lie below and
above $\sigma$, respectively).

Consider a face $f$. Its boundary $\partial f$ has two vertices $v,v'$
such that $\partial f$ is formed by two (directed) paths beginning at $v$
and ending at $v'$; we call $v,v'$ the {\em minimal} and {\em maximal}
vertices of $f$ and denote them by $v^{\min}(f)$ and $v^{\max}(f)$,
respectively. (When $\partial f$ has no common edge with the boundary
$\partial Z$ of $Z$, this fact can be seen by considering the vertex $w$
corresponding to the face $f$ in the RT-diagram $D'$ related to $W'$. Then
one path in $\partial f$ corresponds to the sequence (in the clockwise
order) of the edges in $D'$ leaving $w$, and the other path to the
sequence of edges entering $w$. When $\partial f$ meets an edge in
$\partial Z$, reasonings are easy as well.)

Let $f$ contain a point $x\in\Xdown$. The fact that $H$ is acyclic implies
that there is a path $P=P_0(f)$ in $H$ such that: (a) $P$ begins at a
vertex $v_0$ of $L$, ends at $v^{\min}(f)$ and has all intermediate points
not in $L$, and (b) for any edge $e=(u,v)$ of $P$ and for the other edge
$e'$ of $H$ entering $v$, the index of $e$ is less than the index of $e'$.
Such a path is constructed uniquely. Symmetrically, there is a path
$P'=P_1(f)$ such that: (a') $P'$ begins at a vertex $v_1$ of $L$, ends at
$v^{\min}(f)$ and has all intermediate points not in $L$, and (b') for any
edge $e=(u,v)$ of $P'$ and for the other edge $e'$ of $H$ entering $v$,
the index of $e$ is greater than the index of $e'$.

Clearly $v_0$ occur in $L$ earlier than $v_1$ (and one may say that
$P_0(f)$ goes below $P_1(f)$). Let $\rho(f)$ denote the closed region
bounded by these paths and the path in $L$ from $v_0$ to $v_1$. See the
picture for illustration.
 \begin{center}
  \unitlength=1mm
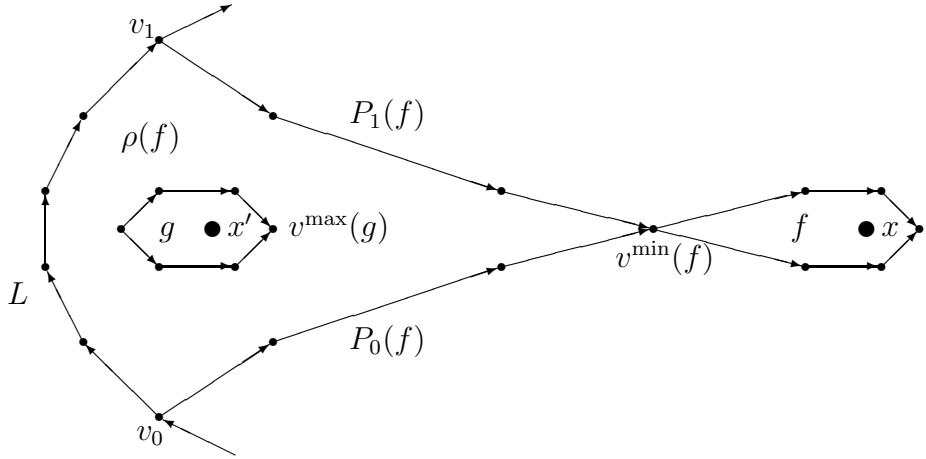
  \begin{picture}(130,60)

  \put(10,25){\circle*{1}}
  \put(10,35){\circle*{1}}
  \put(15,15){\circle*{1}}
  \put(15,45){\circle*{1}}
  \put(25,5){\circle*{1}}
  \put(25,55){\circle*{1}}
  \put(40,15){\circle*{1}}
  \put(40,45){\circle*{1}}
  \put(70,25){\circle*{1}}
  \put(70,35){\circle*{1}}
  \put(90,30){\circle*{1}}
  \put(110,25){\circle*{1}}
  \put(110,35){\circle*{1}}
  \put(120,25){\circle*{1}}
  \put(120,35){\circle*{1}}
  \put(125,30){\circle*{1}}
  \put(20,30){\circle*{1}}
  \put(25,25){\circle*{1}}
  \put(25,35){\circle*{1}}
  \put(35,25){\circle*{1}}
  \put(35,35){\circle*{1}}
  \put(40,30){\circle*{1}}
  \put(35,0){\vector(-2,1){9.5}}
  \put(25,5){\vector(-1,1){9.7}}
  \put(15,15){\vector(-1,2){4.7}}
  \put(10,25){\vector(0,1){9.5}}
  \put(10,35){\vector(1,2){4.7}}
  \put(15,45){\vector(1,1){9.7}}
  \put(25,55){\vector(2,1){9.5}}
  \put(25,5){\vector(3,2){14.5}}
  \put(25,55){\vector(3,-2){14.5}}
  \put(40,15){\vector(3,1){29.5}}
  \put(40,45){\vector(3,-1){29.5}}
  \put(70,25){\vector(4,1){19.5}}
  \put(70,35){\vector(4,-1){19.5}}
  \put(90,30){\vector(4,1){19.5}}
  \put(90,30){\vector(4,-1){19.5}}
  \put(110,25){\vector(1,0){9.5}}
  \put(110,35){\vector(1,0){9.5}}
  \put(120,25){\vector(1,1){4.7}}
  \put(120,35){\vector(1,-1){4.7}}
  \put(20,30){\vector(1,1){4.7}}
  \put(20,30){\vector(1,-1){4.7}}
  \put(25,25){\vector(1,0){9.5}}
  \put(25,35){\vector(1,0){9.5}}
  \put(35,25){\vector(1,1){4.7}}
  \put(35,35){\vector(1,-1){4.7}}
  \put(5,20){$L$}
  \put(25,29){$g$}
  \put(34,29){$x'$}
  \put(32,30){\circle*{2}}
  \put(20,41){$\rho(f)$}
  \put(50,14){$P_0(f)$}
  \put(50,44){$P_1(f)$}
  \put(108,29){$f$}
  \put(120,29){$x$}
  \put(118,30){\circle*{2}}
  \put(85,25){$v^{\min}(f)$}
  \put(42,29){$v^{\max}(g)$}
  \put(22,2){$v_0$}
  \put(21,56){$v_1$}
    \end{picture}
   \end{center}

Let us say that a closed region $\rho$ in $Z$ formed as the union of some
faces and edges of $H$ is an {\em ideal} one if each edge
$e\not\subset\partial Z$ having an end vertex in $\rho$ but the interior
not in $\rho$ is directed from $\rho$ to $Z-\rho$. In view
of~\refeq{clockwise},
  \begin{numitem}
(i) the indices of edges are weakly increasing along the path $P_0(f)$ and
weakly decreasing along $P_1(f)$; and (ii) $\rho(f)$ is ideal.
  \label{eq:region}
  \end{numitem}

We assert that $\rho(f)$ contains no point from $\Xup$. Suppose this is
not so and such a point $x'$ exists. Let $x'$ lie in a face $g$. We are
going to show the existence in $W'$ of an $i$-wire $\sigma$ and a $j$-wire
$\sigma'$ with $i<j$ such that
  \begin{numitem}
$f$ lies below $\sigma$ and above $\sigma'$, while $g$ lies above $\sigma$
and below $\sigma'$.
  \label{eq:2sigma}
  \end{numitem}

To show this, consider the edges $e,e'$ entering $v^{\max}(g)$, and assume
that $e$ is an $i'$-edge, $e'$ is a $j'$-edge, and $i'<j'$. Take the
$i'$-wire $\tau$ containing $e$; then $g$ lies above $\tau$. If $f$ lies
below $\tau$, then we take $\tau$ as the desired $\sigma$
in~\refeq{2sigma}. Now suppose that $f$ lies above $\tau$.
Using~\refeq{clockwise} and~\refeq{region}(i), one can conclude that
$\tau$ meets the path $P_0(f)$, but not $P_1(f)$, at some vertex $v$. Let
$e$ be the edge in $P_0(f)$ beginning at $v$ and take the wire $\sigma$
containing $e$; let it be an $i$-wire. Then~\refeq{clockwise}
and~\refeq{region}(i) imply that $\sigma$ meets $P_1(f)$, whence the face
$f$ lies below $\sigma$. Also $i<i'$, implying that the face $g$ lies
above $\sigma$. A $j$-wire $\sigma'$ required in~\refeq{2sigma} is
constructed in a similar way, starting with the edge $e'$ and using, if
needed, the path $P_1(f)$. By the construction, we have $i<j$.

Since $f$ lies below $\sigma$ and $g$ lies above $\sigma$, we have
$x_i<x'_i$. In its turn, the behavior of $\sigma'$ gives $x_j>x'_j$. Also
$x_k<q\le x'_k$. Note that $j<k$; for if $j=k$ then $\sigma'$ would be of
the form $\sigma_{k,q'}$ with $q'<q$, giving $x'_k<q'<q\le x'_k$. So
$i<j<k$, and we obtain~\refeq{obstacle}; a contradiction.

Thus, the region $\rho(f)$ contains no elements of $\Xup$. Define
  $$
  \hat\rho:=\cup\{\rho(f)\colon f\in F(H)\;\;
  \mbox{$f$ contains an element of $\Xdown$}\}.
  $$
Then $\hat\rho\cap \Xup=\emptyset$. By the construction of regions
$\rho(f)$ and by~\refeq{region}(ii), $\hat\rho\cap\partial Z$ is contained
in the part of $L$ from $p_{1,1}$ to the vertex in $L\cap V(H)$ preceding
$p_{k,q}$. Also
  \begin{numitem}
(i) for each face $f$ with $f\cap\Xdown\ne\emptyset$, the vertex
$v^{\min}(f)$ is in $\hat\rho$; and (ii) $\hat\rho$ is ideal.
  \label{eq:hatrho}
  \end{numitem}

Now we are ready to draw the desired wire $\sigma_{k,q}$. We add to
$\hat\rho$ the part of $L$ from $p_0$ to $p_{k,q}$. Obviously, $p_{k,q}$
belongs to the upmost face $f^{\max}$. Let $U$ be the set of edges
$e\not\subset\partial Z$ connecting $\hat\rho$ and $Z-\hat\rho$; they go
out of $\hat\rho$, by~\refeq{hatrho}(ii). Let $\Fscr$ be the set of faces
$f$ such that $v^{\min}(f)\in \hat\rho$ but $f\not\subseteq\hat\rho$.

First consider the case $q=1$. Then $p'_{k,q}$ belongs to the bottommost
face $f^{\min}$. For each edge $e\in U$, choose a point $v_e$ in the
interior of $e$. Clearly each face $f\in\Fscr-\{f^{\min},f^{\max}\}$ has
exactly two edges $e,e'$ in $U$. We connect $v_e$ and $v_{e'}$ by a curve
$\sigma(f)$ within $f$, subdividing $f$ into two regions (faces). If $f$
has points from $\Xdown$ ($\Xup$), we move them into the region containing
$v^{\min}(f)$ (resp. $v^{\max}(f)$). We act similarly for the face
$f=f^{\max}$ ($f=f^{\min}$) with the only difference that $\sigma(f)$
connects $p_{k,q}$ (resp. $p'_{k,q}$) and $v_e$ for the unique edge $e\in
U$ in $f$. The concatenation of these $\sigma(f)$ for all $f\in\Fscr$
gives the desired wire $\sigma_{k,q}$ (it intersects each wire in $W'$
exactly once, by~\refeq{hatrho}(ii)).

In case $q>1$, we extend $\hat\rho$ by adding to it the region $\rho'$ of
$Z$ formed by the faces lying below $\sigma_{k,q-1}$. Note that no point
of $\Xup$ is contained in $\rho'$  and that no edge of $H$ goes from
$Z-\rho'$ to $\rho'$. Then~\refeq{hatrho} remains valid, and we draw
$\sigma_{k,q}$ as in the previous case.

Add $\sigma_{k,q}$ to $W'$ and continue the process. Upon constructing the
last wire $\sigma_{n,a_n}$, the resulting W-diagram $W$ determines a
required RT-diagram for $X$. (Observe that if a face $f$ for $W$ involves
a point of $X$, then such a point $x$ is unique and the vertex of $D$
corresponding to $f$ is just $\pi(x)$.)
  \end{proof}

 \noindent
 {\bf Remark 3.}
Analyzing the above proof, one can deduce that the height $h(D)$ of the
obtained RT-diagram $D$ is minimum among the possible tilings for $X$.
Moreover, such a $D$ is unique.

\subsection{Sub-zonogons and sub-boxes.} \label{ssec:transf}

We demonstrate one simple consequence of Theorem~\ref{tm:X-tiling} that
will be used later.

Let $y,a'\in\Zset_+^n$ be such that $y+a'\le a$. Then the {\em
sub-zonogon} of $Z(a)$ of size $a'$ with the beginning at $y$ is the set
$Z(y;a'):=y+Z(a')$ (clearly it is contained in $Z(a)$). In other words,
$Z(y;a')$ is the projection by $\pi$ of the convex hull of the sub-box
$B(y|\,y+a'):=\{x\in\Zset_+^n\colon y\le x\le y+a'\}$ of the box $B(a)$
(note that we admit $a_i=0$ for some $i$'s, i.e., the dimension of the
sub-box may be less than $n$).

  \begin{prop} \label{pr:sub-zon}
Any RT-diagram $D'$ for a sub-zonogon $Z(y;a')$ is extendable into an
RT-diagram for $Z(a)$.
  \end{prop}
  \begin{proof}
It suffices to consider the set $X$ of points $x\in B(y|\,y+a')$ such that
$\pi(x)$ lies in the boundary of $Z(y;a')$ and show that~\refeq{problem}
has a solution for this $X$. By Theorem~\ref{tm:X-tiling}, one has to
check that~\refeq{obstacle} does not take place. This is straightforward,
taking into account that for the elements $x\in X$ are coordinate-wise
weakly increasing when $\pi(x)$ moves along the the left boundary of
$Z(y;a)$, and decreasing along the right boundary.
 \end{proof}

(An alternative proof: Assuming that $B':=B(y|\,y+a')$ is not the whole
$B(a)$, there is an $i$ such that the sub-box $B''$ of the form either
$B(\bfzero|\,a-1_i)$ or $B(1_i|\,a)$ includes $B'$. By induction on the
size of the box, one may assume that $D'$ is extendable into an RT-diagram
$D''$ on the sub-zonogon $Z''$ in $Z(a)$ corresponding to $B''$. To extend
$D''$ into an RT-diagram on $Z(a)$ is easy.)

 \medskip
Next, as is explained in the Introduction, any TP-function on a truncated
box can be extended into a TP-function on the entire box. Using
Proposition~\ref{pr:sub-zon}, we can further extend the latter function.
    \begin{prop} \label{pr:a-ap}
Let $B'=B(a'|\,a'')$ be a sub-box of a box $B(a)$ and let $f'$ be a
TP-function on $B'$. Then $f'$ is extendable into a TP-function $f$ on
$B(a)$.
  \end{prop}
  \begin{proof}
Take a normal basis $\Bscr'$ for $B'$, e.g., the standard one, and let
$g'$ be the restriction of $f'$ to $\Bscr'$. Extend the RT-diagram $D'$
corresponding to $\Bscr'$ into an RT-diagram on $Z(a)$. This gives a basis
$\Bscr$ for $B(a)$ including $\Bscr'$. Extend $g'$ arbitrarily into a
function $g$ on $\Bscr$. Then the TP-function $f$ on $B(a)$ determined by
$g$ is as required.
  \end{proof}

 \noindent
 {\bf Remark 4.}
In a recent work, Henriques and Speyer~\cite{HS} present a number of
results on rhombus (viz. rhombic) tilings and their applications to the
$n$-cube recurrence, and others. They argue in direct terms of tilings,
not appealing to wiring diagrams. On this way, they prove
Proposition~\ref{pr:sub-zon} for the case when $Z(y;a')$ is a hexagon.
Also it is shown there that each rhombus tiling can be turned into the
minimal rhombus one by a series of downward flips (viz. transformations of
$\curlywedge$-hexagons into $\curlyvee$-ones), yielding an alternative
proof of Proposition~\ref{pr:diag}. They also consider a certain complex
associated with the set of tilings on a zonogon, prove that it is simply
connected, and use this fact to show that the vertices of tiling index a
basis for the functions obeying the $n$-cube recurrence. We, however, do
not see whether one can apply a similar approach to the case of
TP-functions (since not every TP3-relation concerns a hexagon of a tiling,
as we have seen above).

       \section{Submodular TP-functions} \label{sec:submod}

In this section we consider TP-functions on a box $B(a)$
with the additional property of submodularity. We
demonstrate an important role of the standard basis $Int(a)$ for
such functions by showing that a TP-function is submodular if and only
if its restriction to $Int(a)$ is such. Here $Int(a)$ is the set of
all fuzzy-intervals in $B(a)$ to which the zero fuzzy-interval is
added as well.

Recall that a function $f$ on a lattice $\Lscr$, with meet
operation $\wedge$ and join operation $\vee$, is called {\em
submodular} if it satisfies the {\em submodular inequality}
  $$
  f(\alpha)+f(\beta)\ge f(\alpha\wedge\beta)+f(\alpha\vee\beta)
    $$
for each pair $\alpha,\beta\in\Lscr$.
Sometimes one considers a function $f$ on a part $\Lscr'$ of the lattice,
in which case the submodular inequality is imposed whenever all
$\alpha,\beta,\alpha\vee\beta,\alpha\wedge\beta$ occur is $\Lscr'$.

The lattice operations on elements $x,x'$ of the box $B(a)$ are defined
in a natural way: $x\wedge x'$ and $x\vee
x'$ are the vectors whose $i$th entries are $\min\{x_i,x'_i\}$ and
$\max\{x_i,x'_i\}$, respectively.
A simple fact is that a
function $f$ on the lattice $B(a)$ is submodular if and only if
   \begin{equation} \label{eq:box_sub}
   f(x+1_i)+f(x+1_j)\ge f(x)+f(x+1_i+1_j)
   \end{equation}
holds for all $x,i,j$ ($i\ne j$) such that all four vectors involved
belong to  $B(a)$.

  \begin{theorem} \label{tm:box_sub}
Let $f$ be a TP-function on a box $B(a)$. Then $f$ is submodular
if and only if it is submodular on the standard basis $Int(a)$, where
the latter means that~\refeq{box_sub} holds whenever $i\ne j$ and the
four vectors occurring in it belong to $Int(a)$.
   \end{theorem}
   \begin{proof}
We use results on rhombic tilings from Section~\ref{sec:tiling}.

Consider elements $x,x+1_i,x+1_j,x+1_i+1_j$ of $B(a)$ ($i\ne j$). Their
images in the zonogon $Z(a)$ form a (little) rhombus, and by
(a very special case of) Proposition~\ref{pr:sub-zon}, this rhombus
belongs to some RT-diagram on $Z(a)$. In other words, the above four
elements are contained in some normal basis for $B(a)$. In light of this,
we can reformulate the theorem (and thereby slightly strengthen it)
by asserting that if a TP-function $f$ is
submodular with respect to some normal basis $\Bscr$ (or its corresponding
tiling), then $f$ is submodular w.r.t. any other normal basis.
(When saying that $f$ is {\em submodular w.r.t.} $\Bscr$, we mean
that~\refeq{box_sub} holds whenever the four vectors there belong
to $\Bscr$. The theorem considers as $\Bscr$ the standard basis
$Int(a)$.)

Next, we know (see the proof of Proposition~\ref{pr:diag}) that making
normal mutations (equivalently: transformations of $\curlywedge$-hexagons
into $\curlyvee$-hexagons, or conversely), one can reach any normal basis from
a fixed one. Therefore, it suffices to show that the submodularity
is maintained by a normal mutation.

In other words, it suffices to prove the theorem for the simplest case when
$B(a)$ is the 3-dimensional Boolean cube $C=2^{[3]}$. In this case, the
standard basis $Int$ consists of the sets $\emptyset,1,2,3,12,23,123$, the
submodularity on $Int$ involves the three rhombi of the corresponding tiling,
and one has to check the submodularity for the three rhombi arising under the
mutation $2\rightsquigarrow 13$; see the picture.

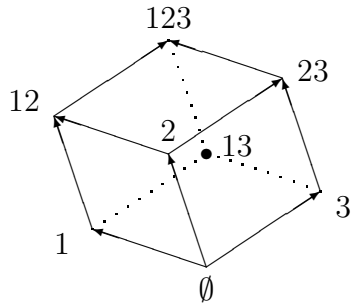
\unitlength=1mm \special{em:linewidth 0.6pt} \linethickness{0.6pt}
 \begin{picture}(78.00,42.00)
\put(60.00,5.00){\vector(3,2){15.00}}
\put(75.00,15.00){\vector(-1,3){5.00}}
\put(70.00,30.00){\vector(-3,1){15.00}}
\put(60.00,5.00){\vector(-3,1){15.00}}
\put(45.00,10.00){\vector(-1,3){5.00}}
\put(40.00,25.00){\vector(3,2){15.00}}
\put(60.00,5.00){\vector(-1,3){5.00}}
\put(55.00,20.00){\vector(3,2){15.00}}
\put(55.00,20.00){\vector(-3,1){15.00}}
\bezier{10}(45.00,10.00)(53.00,15.00)(60.00,20.00)
\bezier{9}(60.00,20.00)(58.00,27.00)(55.00,35.00)
\bezier{9}(60.00,20.00)(67.00,18.00)(75.00,15.00)
\put(60.00,20.00){\circle*{1.500}}
\put(60.00,2.00){\makebox(0,0)[cc]{$\emptyset$}}
\put(78.00,13.00){\makebox(0,0)[cc]{3}}
\put(55.00,23.00){\makebox(0,0)[cc]{2}}
\put(41.00,8.00){\makebox(0,0)[cc]{1}}
\put(36.00,27.00){\makebox(0,0)[cc]{12}}
\put(64.00,21.00){\makebox(0,0)[cc]{13}}
\put(74.00,31.00){\makebox(0,0)[cc]{23}}
\put(55.00,38.00){\makebox(0,0)[cc]{123}}
 \end{picture}

Let $f$ be a TP-function on $C$, i.e., $f$ satisfies
 \begin{equation} \label{TP}
f(2)+f(13)=\max\{f(1)+f(23),f(3)+f(12)\}.
 \end{equation}
The submodularity on $Int$ reads as:
  \begin{eqnarray}
  f(\emptyset)+f(23) &\le& f(2)+f(3); \label{23} \\
   f(\emptyset)+f(12) &\le & f(1)+f(2); \label{12} \\
  f(2)+f(123) &\le& f(12)+f(23). \label{1223}
   \end{eqnarray}

We show that~(\ref{TP})--(\ref{1223}) imply the submodular inequalities
for the other three rhombi, as follows. Adding $f(1)$ to (both sides of)
(\ref{23}) gives
  $$
f(1)+f(23)\le f(2)+f(3)-f(\emptyset)+f(1).
  $$
Adding $f(3)$ to (\ref{12}) gives
  $$
f(3)+f(13)\le f(1)+f(2)-f(\emptyset)+f(3).
  $$
Substituting these inequalities into~(\ref{TP}), we obtain
  $$
f(2)+f(13)=\max(f(1)+f(23),f(3)+f(12))\le
f(1)+f(2)+f(3)-f(\emptyset),
   $$
which implies the submodular inequality for the rhombus
on $\emptyset,1,3,13$:
  $$
f(\emptyset)+f(13)\le f(1)+f(3).
  $$

Arguing similarly, one obtains the submodular inequalities for
the rhombi on $1,12,13,123$ and on $3,13,23,123$. More precisely:
  \begin{eqnarray*}
f(1)+f(123) &\le& f(1)+f(12)+f(23)-f(2) \qquad\qquad
                            \mbox{(by~(\ref{1223}))} \\
  &\le& f(2)+f(13)+f(12)-f(2) \qquad\qquad
                        \mbox{(by~(\ref{TP}))} \\
  &=& f(13)+f(12);
  \end{eqnarray*}
and
  \begin{eqnarray*}
f(3)+f(123) &\le& f(3)+f(12)+f(23)-f(2) \qquad\qquad
                        \mbox{(by~(\ref{1223}))} \\
   &\le& f(2)+f(13)+f(23)-f(2)  \qquad\qquad
                  \mbox{(by~(\ref{TP}))} \\
   &=& f(13)+f(23).
  \end{eqnarray*}

(Note that if needed, one can reverse the arguments to
obtain~(\ref{23})--(\ref{1223}) from the other three inequalities.)
  \end{proof}

 \noindent
{\bf Remark 5.} If we replace in Theorem~\ref{tm:box_sub} the
submodularity condition by the corresponding {\em supermodularity}
condition (i.e., replace $\ge$ by $\le$), then the TP-function $f$ need
not be supermodular globally, even in the Boolean case with $n=3$. A
counterexample is the function on $2^{[3]}$ taking value 0 on
$\{\emptyset\},1,2,3,12$ and value 1 on $13,23,123$ (the supermodularity
is violated for the sets $13$ and $23$). On the other hand, one can show
that a version of the theorem concerning {\em modular} TP-functions is
valid.


\section{Skew-submodular TP-functions} \label{sec:skew}

In this section we show that another important property can also be
TP-propagated from the standard basis to the entire box.

 \smallskip
 \noindent
{\bf Definition.} We say that a function $f$ on a box $B(a)$ is {\em
skew-submodular} if
\begin{equation} \label{skew}
   f(x+1_i+1_j)+f(x+1_j) \ge f(x+1_i)+f(x+2_j)
\end{equation}
holds for all $x,i,j$, $i\ne j$, such that all four vectors involved are
in $B(a)$.

 \smallskip
Here $2_j$ stands for $2 \cdot 1_j$, and note that $i,j$ need not be
ordered. So the skew-submodularity imposes a restriction on $f$ within
each sub-box of the form $B(x|\,x+1_i+2_j)$ in $B(a)$. The picture
illustrates the corresponding tiling of the zonogon $Z(1_i+1_j)$ when
$i<j$ (on the right) and $j<i$ (on the left); here the skew-submodular
condition reads as $f(B)+f(C)\ge f(A)+f(D)$.

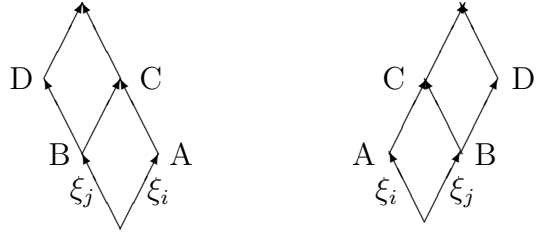
\unitlength=1mm \special{em:linewidth 0.6pt} \linethickness{0.6pt}
\begin{picture}(93.00,38.00)
\put(40.00,5.00){\vector(1,2){5.00}}
\put(45.00,15.00){\vector(-1,2){5.00}}
\put(40.00,25.00){\vector(-1,2){5.00}}
\put(40.00,5.00){\vector(-1,2){5.00}}
\put(35.00,15.00){\vector(-1,2){5.00}}
\put(30.00,25.00){\vector(1,2){5.00}}
\put(35.00,15.00){\vector(1,2){5.00}}
\put(80.00,6.00){\vector(-1,2){4.67}}
\put(75.33,15.00){\vector(1,2){5.00}}
\put(80.33,25.00){\vector(1,2){5.00}}
 \put(80.00,6.00){\vector(1,2){4.67}}
\put(84.67,15.00){\vector(1,2){5.00}}
\put(89.67,25.00){\vector(-1,2){5.00}}
\put(85.00,15.00){\vector(-1,2){5.00}}
\put(45.00,10.00){\makebox(0,0)[cc]{$\xi_i$}}
\put(35.00,10.00){\makebox(0,0)[cc]{$\xi_j$}}
\put(75.00,10.00){\makebox(0,0)[cc]{$\xi_i$}}
\put(85.00,10.00){\makebox(0,0)[cc]{$\xi_j$}}
\put(72.00,15.00){\makebox(0,0)[cc]{A}}
\put(88.00,15.00){\makebox(0,0)[cc]{B}}
\put(76.00,25.00){\makebox(0,0)[cc]{C}}
\put(93.00,25.00){\makebox(0,0)[cc]{D}}
\put(48.00,15.00){\makebox(0,0)[cc]{A}}
\put(32.00,15.00){\makebox(0,0)[cc]{B}}
\put(44.00,25.00){\makebox(0,0)[cc]{C}}
\put(27.00,25.00){\makebox(0,0)[cc]{D}}
\end{picture}

In fact, one can regard~(\ref{skew}) as a degenerate form of the
TP3-relation~\refeq{GP3}. Indeed, putting $j=k$ in~\refeq{GP3}, we obtain
 $$
f(x+1_i+1_j)+f(x+1_j)=\max\{f(x+1_i+1_j)+f(x+1_j), f(x+2_j)+f(x+1_i)\},
 $$
which is just equivalent to~(\ref{skew}).

 \begin{theorem} \label{tm:skew}
A TP-function $f$ on a box $B(a)$ is skew-submodular if and only if its
restriction to the standard basis $Int(a)$ is skew-submodular (in the
sense that~(\ref{skew}) holds whenever $i\ne j$ and the four vectors
occurring in it belong to $Int(a)$). Furthermore, a skew-submodular $f$
satisfies the additional relations:
   \begin{equation} \label{slope}
 f(x+1_i+1_j)+f(x+1_j+1_k) \ge f(x+1_i+1_k)+f(x+2_j),
  \end{equation}
where $i,j,k$ are different.
  \end{theorem}
  \begin{proof}
Arguing as in the previous section and using Propositions~\ref{pr:diag}
and~\ref{pr:sub-zon}, we reduce the task to examination of the
3-dimensional boxes $B(1,1,2)$, $B(1,2,1)$ and $B(2,1,1)$. Below we
consider the case $B(1,2,1)$ (in the other two cases, the proof is
analogous and we leave it to the reader as an exercise). This case is
illustrated in the picture:

\unitlength=1mm \special{em:linewidth 0.6pt} \linethickness{0.6pt}
\begin{picture}(84.00,55.00)
\put(65.00,5.00){\vector(3,2){15.00}}
\put(80.00,15.00){\vector(-1,3){5.00}}
\put(75.00,30.00){\vector(-1,3){5.00}}
\put(70.00,45.00){\vector(-3,1){15.00}}
\put(65.00,5.00){\vector(-1,3){5.00}}
\put(60.00,20.00){\vector(-1,3){5.00}}
\put(55.00,35.00){\vector(3,2){15.00}}
\put(65.00,5.00){\vector(-3,1){15.00}}
\put(50.00,10.00){\vector(-1,3){5.00}}
\put(45.00,25.00){\vector(-1,3){5.00}}
\put(40.00,40.00){\vector(3,2){15.00}}
\put(60.00,20.00){\vector(3,2){15.00}}
\put(60.00,20.00){\vector(-3,1){15.00}}
\put(55.00,35.00){\vector(-3,1){15.00}}
\bezier{10}(50.00,10.00)(57.00,15.00)(65.00,20.00)
\bezier{8}(65.00,20.00)(72.00,17.00)(80.00,15.00)
\bezier{14}(65.00,20.00)(60.00,35.00)(55.00,50.00)
\bezier{8}(60.00,35.00)(68.00,32.00)(75.00,30.00)
\bezier{10}(60.00,35.00)(53.00,30.00)(45.00,25.00)
\put(60.00,35.00){\circle*{1.5}}
 \put(65.00,20.00){\circle*{1.5}}
\put(65.00,2.00){\makebox(0,0)[cc]{$A$}}
\put(61.00,23.00){\makebox(0,0)[cc]{$A'$}}
\put(55.00,38.00){\makebox(0,0)[cc]{$A''$}}
\put(47.00,8.00){\makebox(0,0)[cc]{$B$}}
\put(42.00,23.00){\makebox(0,0)[cc]{$B'$}}
\put(36.00,39.00){\makebox(0,0)[cc]{$B''$}}
\put(68.00,21.00){\makebox(0,0)[cc]{$C$}}
\put(63.00,37.00){\makebox(0,0)[cc]{$C'$}}
\put(56.00,53.00){\makebox(0,0)[cc]{$C''$}}
\put(83.00,16.00){\makebox(0,0)[cc]{$D$}}
\put(78.00,32.00){\makebox(0,0)[cc]{$D'$}}
\put(73.00,47.00){\makebox(0,0)[cc]{$D''$}}
 \put(56,6){\makebox(0,0)[cc]{$\xi_1$}}
 \put(65,14){\makebox(0,0)[cc]{$\xi_2$}}
 \put(75,8){\makebox(0,0)[cc]{$\xi_3$}}
 \end{picture}

There are two TP3-relations in the box, namely:
   \begin{equation}\label{*}
f(A')+f(C)=\max(f(B)+f(D'),f(B')+f(D))
  \end{equation}
and
  \begin{equation}\label{**}
f(A'')+f(C')=\max(f(B')+f(D''),f(B'')+f(D')).
  \end{equation}
The face (parallelogram) $AA''B''B$ gives the skew-submodular inequality
in the standard basis:
  \begin{equation}\label{ab}
  f(A'')+f(B)\le f(A')+f(B').
  \end{equation}
The face $AA''D''D$ gives one more skew-submodular inequality
\begin{equation}\label{ad}
  f(D)+f(A'')\le f(D')+f(A').
\end{equation}

First of all we prove inequality (\ref{slope}) (with $(i,j,k)=(1,2,3)$);
it is viewed as
  \begin{equation}\label{abcd}
f(B')+f(D')\ge f(A'')+f(C).
  \end{equation}
Adding $f(D')$ to (\ref{ab}) gives
  $$
f(A'')+f(B)+f(D')\le f(A')+f(B')+f(D').
  $$
Adding $f(B')$ to (\ref{ad}) gives
  $$
f(D)+f(A'')+f(B')\le f(A')+f(D')+f(B').
  $$
These inequalities together with~(\ref{*}) result in
$$
f(A'')+f(A')+f(C)\le f(A')+f(B')+f(D').
$$
Now the desired inequality~(\ref{abcd}) is obtained by canceling $f(A')$
in both sides.

Next we show validity of the other two skew-submodular inequalities in the
box, namely, those concerning the faces $BB''C''C$ and $DD''C''C$.

Adding (\ref{abcd}) and the inequality $f(B'')+f(D')\le f(A'')+f(C')$
(which is a consequence of (\ref{**})), we obtain
  $$
f(A'')+f(C)+f(B'')+f(D')\le f(B')+f(D')+f(A'')+f(C').
  $$
Canceling $f(A'')+f(D')$ in this inequality gives
  $$
f(C)+f(B'')\le f(B')+f(C'),
  $$
which is just the skew-submodular inequality for the face $BB''C''C$. The
skew-submodular inequality $f(C)+f(D'')\le f(D')+f(C')$ for the face
$DD''C''C$ is obtained in a similar way.
  \end{proof}


\section{Discrete concave TP-functions} \label{sec:concave}

In this section we combine the above submodular and skew-submodular
conditions on TP-functions.

Let us say that a TP-function $f$ on a box $B(a)$ is a {\em DCTP-function}
if
 \begin{equation}\label{DC}
f(x+1_i+1_j)+f(x+1_j+1_k) \ge f(x+2_j)+f(x+1_i+1_k)
 \end{equation}
holds for all $x\in B(a)$ and $i,j,k\in \{0\}\cup [n]$ such that the four
vectors in this relation belong to $B(a)$. Here $1_0$ means the zero
vector. Note that $i,j,k$ need not be ordered and some of them may
coincide.

  \smallskip
 \noindent
{\bf Remark 6.} The meaning of the abbreviation ``DC'' is that the
TP-functions obeying~(\ref{DC}) possess the property of discrete
concavity. More precisely, one can check that such functions satisfy
requirements in a discrete concavity theorem from~\cite[Ch.~6]{M}, and
therefore, they form a subclass of {\em polymatroidal concave functions},
or $M^\#$-concave functions, in terminology of that paper.

 \smallskip
Observe that if $j=0\ne i,k$ and $i\ne k$, then~(\ref{DC}) turns into the
submodular condition (cf.~\refeq{box_sub}). If $k=0\ne i,j$ and $i\ne j$,
then~(\ref{DC}) turns into the skew-submodular condition~(\ref{skew}). And
if $i=k=0$, then (\ref{DC}) turns into the concavity inequality
  $$
2f(x+1_j)\ge f(x)+f(x+2_j).
  $$
One easily shows that this inequality follows from submodular and
skew-submodular relations.

Now assume that none of $i,j,k$ is $0$. If all $i,j,k$ are different,
then~(\ref{DC}) is a consequence of the skew-submodularity, due to
Theorem~\ref{tm:skew}. Finally, if $i=k$, then~(\ref{DC}) turns into
  $$
2f(x+1_i+1_j) \ge f(x+2_i)+f(x+2_j),
  $$
which again is easily shown to follow from skew-submodular relations.

The above observations are summarized as follows.
  \begin{prop}  \label{pr:DC}
A TP-function on a box is a DCTP-function if and only if it is submodular
and skew-submodular.
  \end{prop}

This proposition and Theorems~\ref{tm:box_sub} and~\ref{tm:skew} give the
following
  \begin{corollary}  \label{cor:DC}
A TP-function $f$ on a box $B(a)$ is a DCTP-function if and only if it is
submodular and skew-submodular on the standard basis $Int(a)$.
  \end{corollary}

One can visualize this corollary by considering the standard tiling of the
zonogon $Z(a)$ (i.e., the RT-diagram associated with the standard basis
$Int(a)$). It contains ``big'' parallelograms $P(i,j)$ for $i<j$, where
$P(i,j)$ is the sub-zonogon $Z(a_{i+1}\xi_{i+1}+\ldots+a_{j-1}\xi_{j-1}\ ;
\ a_i1_i+a_j1_j)$. Subdivide each $ij$-rhombus
$[x,x+\xi_i,x+\xi_j,x+\xi_i+\xi_j]$ in $P(i,j)$ into two triangles by
drawing the diagonal $[x+\xi_i,x+\xi_j]$. This gives a triangulation of
$P(i,j)$; see the picture where $a_i=2$ and $a_j=3$.

 \unitlength=1mm \special{em:linewidth 0.6pt} \linethickness{0.6pt}
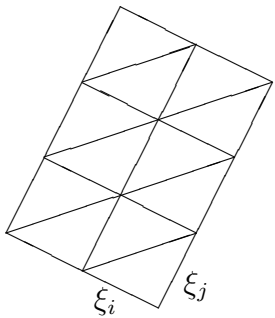
 \begin{picture}(75.00,45)
 \put(60,5){\line(1,2){15}}
 \put(50,10){\line(1,2){15}}
 \put(40,15){\line(1,2){15}}
 \put(60,5){\line(-2,1){20}}
 \put(65,15){\line(-2,1){20}}
 \put(70,25){\line(-2,1){20}}
 \put(75,35){\line(-2,1){20}}
 \put(50,10){\line(3,1){15}}
 \put(40,15){\line(3,1){30}}
 \put(45,25){\line(3,1){30}}
 \put(50,35){\line(3,1){15}}
 \put(65.00,8.00){\makebox(0,0)[cc]{$\xi_j$}}
 \put(53.00,6.00){\makebox(0,0)[cc]{$\xi_i$}}
 \end{picture}

In terms of such triangulations, the submodular and skew-submodular
conditions on $f$ say that for any two adjacent triangles $ABC$ and $BCD$,
one holds $f(B)+f(C)\ge f(A)+f(D)$. In other words, the affine
interpolation of $f$ within each parallelogram $P(i,j)$ is concave.

 \medskip
Next we use the above description of DCTP-functions on $B(a)$ to obtain a
characterization of DCTP-functions on a box truncated from above.
   \begin{prop}  \label{pr:trun}
A TP-function $f$ on a truncated box $B_0^{m'}(a)$ is a DCTP-function if
and only if it is submodular and skew-submodular on the standard basis
$\mathcal B=Int(a;1)\cup\ldots Int(a;m')$.
  \end{prop}
   \begin{proof}
For convenience we identify the elements of $B_0^{m'}(a)$ with their
projections in the zonogon $Z(a)$. Then the basis $\mathcal B$ consists of
the vertices $x$ of the standard tiling of $Z(a)$ such that $|x|\le m'$.
The submodularity and skew-submodularity of $f$ on $\Bscr$ means that, in
the triangulation of each parallelogram $P(i,j)$ as above, the inequality
$f(B)+f(C)\ge f(A)+f(D)$ holds for every two adjacent triangles $ABC$ and
$BCD$ of the triangulation, with all $A,B,C,D$ occurring in $\Bscr$. We
assert that the restriction of $f$ to $\mathcal B$ can be extended to a
function $\tilde f$ on the standard basis $Int(a)$ of the entire box
$B(a)$ such that $\tilde f$ is submodular and skew-submodular.

With respect to $\mathcal B$, there are three groups of ``big''
parallelograms of the standard tiling of $Z(a)$. The first group consists
of those $P(i,j)$ which have all vertices in $\mathcal B$, the second one
consists of the $P(i,j)$'s having vertices in $\mathcal B$ and not in
$\mathcal B$, and the third one consists of the $P(i,j)$'s with all
vertices not in $\mathcal B$.

We order the parallelograms of the second group clockwise; then two
consecutive parallelograms share a common edge. The following claim will
be of use.

 \medskip
 \noindent
{\bf Claim}. {\em Let $h$ be a discrete concave function on a
2-dimensional truncated box $B_0^c(0|\, (a,b))$. When $c<b$, let $g$ be a
concave function on the segment $[c,b]$ such that $h(0,c)=g(c)$ and
$h(0,c-1)-h(0,c)\ge g(c)-g(c+1)$. Then there exists a discrete concave
function $\tilde h:B (0|\, (a,b))\to\mathbb R$ such that $\tilde h$
coincides with $h$ on $B_0^c(0| (a,b))$ and $\tilde h(0,k)=g(k)$ for all
$k\in [c,b]$. Symmetrically, when $c>b$, let $g$ be a concave function on
the segment $[c-b,a]$ such that $h(c-b,b)=g(c-b)$ and $h(c-b-1,
b)-h(c-b,b)\ge g(c-b)-g(c-b+1)$. Then there exists a discrete concave
function $\tilde h:B (0 |\,(a,b))\to\mathbb R$ such that $\tilde h$
coincides with $h$ on $B_0^c(0 |\,(a,b))$ and $\tilde h(k,b)=g(k)$ for all
$k\in [c-b,a]$. }

 \medskip
 \noindent
{\em Proof of the Claim.} Consider the case $c<b$. W.l.o.g., we may assume
that $h$ is monotone (otherwise we add to $h$ an appropriate separable
discrete concave function). Define $\tilde h$ by
  $$
\tilde h=h*\tilde g*\phi\{x\ge 0\},
  $$
where: $*$ denotes the convolution of two functions (namely,
$f_1*f_2(z)=\max_{x+y=z}\{f_1(x)+f_2(y)\}$); $\phi\{x\ge 0\}:=0$ for $x\ge
0$ and $-\infty$ otherwise; and $\tilde g=g*\phi\{y\ge 0\}$.

Then $\tilde h$ is a discrete concave function (as the convolution of
discrete concave functions is discrete concave as well; see,
e.g.,~\cite{M}). Obviously, $\tilde h(0,k)=g(k)$. By the discrete
concavity of $h$, we have $h(x,c-x-1)-h(x,c-x)\ge h(0,c-1)-h(0,c)$, and
since $h(0,c-1)-h(0,c)\ge g(c)-g(c+1)$, we obtain $h(x,c-x-1)-h(x,c-x)\ge
g(c)-g(c+1)$. This implies that $\tilde h$ coincides with $h$ on
$B_0^c(0|\,(a,b))$. In case $c>b$, we argue in a similar way. \qed

\smallskip
Using the Claim, we extend $f$, step-by-step in the clockwise order, into
a function which is discrete concave on each parallelogram of the second
group.

Now consider a wiring associated to the big parallelograms (cf.
Section~\ref{sec:tiling}). Then the parallelograms of the third group
induce a connected sub-wiring. The extension of $f$ to the parallelograms
of the second group assigns (1-dimensional) discrete concave boundary
values to the parallelograms of the third group, at most one boundary
function being assigned to each wire. Thus, we can extend these boundary
values to a separable discrete concave function on each parallelogram of
the third group.

This completes the proof of the proposition. \qed
  \end{proof}

As a consequence of this proposition, we obtain the following result; it
was announced without proof in~\cite{HS}.
  \begin{corollary} \label{cor:coro}
A TP-function $f$ is discrete concave on a simplex $B_m^{m}(m^n)$ if and
only if it is submodular and skew-submodular on the standard basis
$\mathcal B=Sint( m^n;m)\cup Int(m^n;m)$.
  \end{corollary}
  \begin{proof}
Consider the projection of $B_m^{m}( m^n)$ into $\mathbb R^{n-1}$ along
the first coordinate. Then $f$ becomes a TP-function on the truncated box
$B_0^m(m^{n-1}))$, and $\mathcal B$ is projected to $Int(m^{n-1};1)\cup
\ldots\cup Int(m^{n-1};m)$. Now the result follows from
Proposition~\ref{pr:trun}.
  \end{proof}

\section{The tropical Laurent phenomenon} \label{sec:Laurent}

A collection of functions on a set $\Xscr$ is said to possess the {\em
Laurentness property} w.r.t. a subset $\Bscr\subset\Xscr$ if the values of
these functions on the elements in $\Xscr-\Bscr$ are expressed as Laurent
polynomials (depending on elements but not on functions) in the values on
$\Bscr$, whereas the latter values are usually assumed to be
``independent''. For the Laurent phenomenon under the octahedron and cube
recurrences, see~\cite{FZ,HS,Sp}.

Its tropical analogue, the {\em tropical Laurentness property}, means that
the value of a function $f$ on an element $x\in\Xscr-\Bscr$ is expressed
as
  $$
  f(x)=\max_{i=1,\ldots,N} \sum\nolimits_{y\in\Bscr}
    h_{i,y}f(y),
  $$
where the coefficients $h_{i,y}$ are integers depending on $x$ (but not
on $f$). In other words, $f(x)$
is represented by a piece-wise linear convex function of which arguments
are the values of $f$ on the elements of $\Bscr$.

In what follows we explain that the TP-functions possess the tropical
Laurentness property w.r.t. the standard basis, and moreover, estimate
the coefficients in the corresponding ``tropical Laurent polynomials''.

 \medskip
 \noindent
{\bf 1.}
We start with the TP-functions on the cube $C=2^{[n]}$. The fact that
$\Tscr(C)$ possesses the tropical Laurentness property w.r.t. the
standard basis $Int$ (consisting of the intervals in $[n]$) easily
follows from results in Section~\ref{sec:boolean}.

Indeed, by Theorem~\ref{tm:flow} and Proposition~\ref{pr:matr}, a function
$f\in\Tscr(C)$ one-to-one corresponds to an $n\times n$ weight matrix $W$ as
in~\refeq{Wp} (in our case $m=0$ and $W=W'$), and the value of $f$ on a set
$S\subseteq[n]$ is viewed as
   \begin{equation}  \label{eq:fS}
   f(S)=\max\{w(\Fscr)\colon \Fscr\in\Phi_S\},
   \end{equation}
where $\Phi_S$ is the set of admissible flows for $S$ (i.e., beginning
at the source set $\{s_p\colon p\in S\}$); cf.~\refeq{f-w}. Notice that
each flow $\Fscr$ in this expression is essential. Indeed, if we put
$w_{pq}:=1$ for all vertices $v_{pq}$ with $p\ge q$ covered by $\Fscr$,
and $w_{pq}:=0$ for the other vertices in the grid $\Gamma_{n,n}$, then
$\Fscr$ is the unique maximum-weight flow in $\Phi_S$ for this matrix $W$.
So the number of linear pieces (slopes) in~\refeq{fS} is just $|\Phi_S|$.

According to~\refeq{g-Wp}, the weight $w_{pq}$ of each vertex $v_{pq}$
can be expressed, by a linear form, via the values of $f$ on intervals:
   \begin{equation}  \label{eq:w-int}
   w_{pq}=\sum\nolimits_{I\in Int} h_{pq}(I)f(I),
   \end{equation}
where each coefficient $h_{pq}(I)$ is 0,1 or --1 ($h_{pq}$ is zero when
$p<q$). Taking the sum of weights $w_{pq}$ over the set
$\Pi(\Fscr)$ of pairs $pq$ concerning $\Fscr$ and substituting it
into~\refeq{fS}, we obtain the desired tropical Laurent polynomial:
  \begin{equation} \label{eq:laurent}
  f(S)=\max\left\{ \sum\nolimits_{I\in Int} h_\Fscr(I)f(I)\colon
              \Fscr\in \Phi_S \right\} ,
  \end{equation}
where $h_\Fscr:=\sum(h_{pq}\colon pq\in \Pi(\Fscr))$.

We assert that all coefficients $h_\Fscr(I)$ are between --1 and 2.

To show this, consider a path $P$ in $\Fscr$. For an intermediate vertex
$v=v_{pq}$ of $P$, we say that $P$ makes {\em right turn} at $v$ if the edge
$e$ of $P$ entering $v$ is horizontal (i.e., $e=(v_{p+1,q},v)$) while the
edge $e'$ leaving $v$ is vertical (i.e., $e'=(v,v_{q+1})$), and say that
$P$ makes {\em left turn} at $v$ if $e$ is vertical while $e'$ is
horizontal. Also if the first edge of $P$ is horizontal, we (conditionally)
say that $P$ makes left turn at its beginning vertex as well. Let $h_P$
denote the sum of functions $h_{pq}$ over the vertices $v_{pq}$
contained in $P$. The values of $h_P$ on the intervals can be calculated
by considering relations in~\refeq{g-Wp} and making corresponding
cancelations when moving along the path $P$. More precisely, one can see
that
  \begin{numitem}
(i) if $P$ makes left turn at $v_{pq}$, then $h_P([p-q+1..p])=1$ and
$h_P([p-q+1..p-1])=-1$ (unless $q=1$, in which case the interval
$[p-q+1..p-1]$ vanishes); (ii) if $P$ makes right turn at $v_{pq}$,
then $h_P([p-q+1..p])=-1$ and $h_P([p-q+1..p-1])=1$; and (iii) $h_P(I)=0$
for the remaining intervals $I$ in $[n]$.
  \label{eq:8path}
  \end{numitem}

This enables us to estimate the values of $h_\Fscr$, i.e., of the sum of
the functions $h_P$ over the paths $P$ in $\Fscr$. Consider an interval
$I=[c..d]$. Since the paths in $\Fscr$ are disjoint, \refeq{8path} shows
that there are at most two paths $P$ such that $h_P(I)\ne 0$. Therefore,
$|h_\Fscr(I)|\le 2$. Suppose $h_\Fscr(I)=-2$. Then $h_P(I)=h_{P'}(I)=-1$
for some (neighboring) paths $P,P'$ in $\Fscr$. In view of~\refeq{8path},
this can happen only if one of these paths makes right turn at the vertex
$v_{pq}$ with $p=d$ and $q=d-c+1$, while the other path makes left turn at
the vertex $v_{p+1,q+1}$. But then $P,P'$ must intersect; a contradiction.
See Fig.~\ref{fig:turns}(a).

\begin{figure}[htb]
 \begin{center}
  \unitlength=1mm
  \begin{picture}(110,20)(0,0)
  \put(30,5){\circle*{1}}
  \put(30,15){\circle*{1}}
  \put(40,5){\circle*{1}}
  \put(40,15){\circle*{1}}
  \put(30,5){\circle{2.5}}
   \put(37,5){\vector(-1,0){6.5}}
   \put(39.5,15){\vector(-1,0){6.5}}
   \put(30,5.5){\vector(0,1){6.5}}
   \put(40,8){\vector(0,1){6.5}}
  \put(10,0){(a)}
  \put(25,2){$v_{pq}$}
  \put(25,9){$P$}
  \put(42,9){$P'$}
  \put(41,16){$v_{p+1,q+1}$}
  \put(90,5){\circle*{1}}
  \put(90,15){\circle*{1}}
  \put(100,5){\circle*{1}}
  \put(100,15){\circle*{1}}
  \put(90,5){\circle{2.5}}
   \put(90.5,5){\vector(-1,0){6.5}}
   \put(107,15){\vector(-1,0){6.5}}
   \put(90,-2){\vector(0,1){6.5}}
   \put(100,15.5){\vector(0,1){6.5}}
  \put(70,0){(b)}
  \put(84,-1){$P$}
  \put(103,18){$P'$}
  \end{picture}
 \end{center}
\caption{(a) $h_\Fscr(I=[p-q+1..p])=-2$; \;\;\; (b)
$h_\Fscr([p-q+1..p])=2$.}
 \label{fig:turns}
  \end{figure}

Thus, $-1\le h_\Fscr(I)\le 2$, as required. (In fact, $h_\Fscr(I)=2$ is
possible; in this case there are two paths in $\Fscr$, one making left
turn at $v_{d,d-c+1}$, and the other making right turn at $v_{d+1,d-c+2}$.
See Fig.~\ref{fig:turns}(b).) Summing up the above reasonings, we obtain
the following
   \begin{prop} \label{pr:laur-cube}
The set of TP-function on the cube $2^{[n]}$ possesses the tropical
Laurentness property w.r.t. $Int$. This is expressed by~\refeq{laurent}
for each $S\in 2^{[n]}-Int$. Moreover, each coefficient $h_\Fscr(I)$ in
this expression is in $\{-1,0,1,2\}$.
   \end{prop}

(Note that the lower and upper bounds --1 and 2 on the ``tropical
monomial'' coefficients in this expression are similar to those on the
exponents of face variables established by Speyer and stated in the Main
Theorem of~\cite{Sp}, where algebraic Laurent polynomials are considered.)

 \medskip
 \noindent
{\bf Remark 7.} Adding an appropriate expression to each sum in the
maximum, one can re-write~\refeq{laurent} in the form
  \begin{multline*}
  f(S)=\max\left\{ \sum\nolimits_{I\in Int} h'_{\Fscr}(I)f(I)
   \colon \Fscr\in\Phi_S \right\}        \\
   -\sum(f(I)\colon I\in Int,\, I\subseteq[\min(S)+1..\max(S)-1]),
    \end{multline*}
where all coefficients $h'_\Fscr(I)$ are nonnegative integers not
exceeding 3. For example, for 2-element sets $ik=\{i,k\}$, $i<k$, one can
obtain
  \begin{multline*}
  f(ik)=\max\nolimits_{i<j<k}\{f(i)+\ldots+f(j-1)+f(j,j+1)+f(j+2)
                  +\ldots +f(k)\} \\
    -f(i+1)-\ldots -f(k-1).
   \end{multline*}

 \noindent
{\bf Remark 8.} The admissible flows figured in~\refeq{laurent} can be
replaced by somewhat more transparent objects. Let us say that a
triangular array $A=(a_{ij})$, $1\le j\le i\le k$, of size $k$ is a {\em
semi-strict Gelfand-Tsetlin pattern} if $a_{i,j-1}<a_{ij}\le a_{i+1,j-1}$
holds for all $i,j$. (Classical Gelfand-Tsetlin patterns~\cite{GT}, or
GT-patterns, are defined by the non-strict inequalities in both sides.)
The tuple $a_{11}<a_{21}<\ldots<a_{k1}$ is called the {\em shape} of $A$.
For each $S\subseteq[n]$, there is a bijection between the set $\Phi_S$ of
admissible flows for $S$ and the set of semi-strict GT-patterns of size
$|S|$ with the shape $p_1<\ldots<p_{|S|}$, where $S=\{p_1,\ldots,
p_{|S|}\}$.

Indeed, given $\Fscr\in\Phi_S$, let $P_i$ be the path in $\Fscr$ beginning
at $s_{p_i}$. Let $V_i$ be the set of vertices entered by vertical edges
of $P_i$ plus the source $s_{p_i}$. The second coordinate of the vertices
in $V_i$ runs from 1 through $i$ (along $P_i$) and we denote these
vertices as $v_{a_{ij},j}$, $j=1,\ldots,i$. Then the admissibility of
$\Fscr$ implies that the arising triangular array $(a_{ij})$ (of size
$|S|$) is a semi-strict GT-pattern. Conversely, given a semi-strict
pattern $A$ of size $k$ with $a_{k1}\le n$, one can uniquely construct an
admissible flow $\Fscr$ in which the vertices entered by vertical edges
are just $v_{a_{ij},j}$ for $i=1,\ldots,k$ and $j=2,\ldots,i$, and the
sources are $v_{a_{i1},1}=s_{a_{i1}}$, $i=1, \ldots,|S|$.

For a semi-strict GT-pattern $A$ of shape $p_1<\ldots<p_k\le n$ and a
TP-function $f$ on $2^{[n]}$, define
   $$
 \hat f(A):=\sum\nolimits_{i,j} \Delta f([a_{ij}-j+1..a_{ij}]),
   $$
where for an interval $I=[c..d]$,
  $$
  \Delta f(I):=f(I)+f(I-\{c,d\})-f(I-\{c\})-f(I-\{d\})
  $$
if $c<d$, and $\Delta f(I):=f(I)$ if $c=d$ (assuming $f(\emptyset)=0$).
One can check that for  an admissible flow $\Fscr$ and its corresponding
semi-strict GT-pattern $A$, $\hat f(A)$ is equivalent to
$\sum_{I\in Int} h_\Fscr(I)$. This and Proposition~\ref{pr:laur-cube}
give the following
  \begin{corollary}  \label{cor:laur-GT}
For a TP-function $f$ on $2^{[n]}$ and a subset $S=\{p_1,\ldots,p_{|S|}\}
\subseteq[n]$ with $p_1<\ldots<p_{|S|}$, one holds
  $$
  f(S)=\max\ \hat f(A),
  $$
where the maximum is taken over all semi-strict GT-patterns $A$ with the
shape $p_1<\ldots< p_{|S|}$.
  \end{corollary}

 \noindent
{\bf 2.} Next we consider a truncated cube $C_m^{m'}\subset 2^{[n]}$. In
this case the tropical Laurentness property for the TP-functions w.r.t.
the standard basis $\Bscr$ is shown in a similar way as for the entire
cube $2^{[n]}$. There are only two differences. The first one is that
expression~\refeq{w-int} should be modified if $q\le m$ (and $p>m$). Now
the weight $w_{pq}$ is expressed by involving corresponding sesquialteral
intervals according to the map $\pi$ in~\refeq{g-Wp}. The second
difference is that for each set $S\in C_m^{m'}-\Bscr$, one should consider
only those admissible flows $\Fscr$ for $S$ that cover all vertices
$v_{pq}$ with $p,q\le m$ (for otherwise $w(\Fscr)$ tends to $-\infty$ when
the number $M$ in~\refeq{Wpp} increases). We call such a flow {\em strong}
and denote the set of strong flows for $S$ by $\Phi_S^{{\rm st}}$. Then
the statement in Proposition~\ref{pr:laur-cube} is modified as follows.
  \begin{prop}  \label{pr:laur-trunc}
For the TP-functions $f$ on $C_m^{m'}$,
   $$
 f(S)=\max\left\{ \sum\nolimits_{X\in\Bscr} h_\Fscr(X)f(X)\colon
              \Fscr\in \Phi_S^{{\rm st}}\right\}
   $$
holds for each set $S\in C_m^{m'}$, where $\Bscr$ is the standard basis for
$C_m^{m'}$.
Also each coefficient $h_\Fscr(X)$ in this expression is in
$\{-1,0,1,2\}$.
  \end{prop}

 \noindent
{\bf 3.} Finally, we consider an $n$-dimensional box $B(a)$. To show the
Laurentness property for $\Tscr(B(a))$ w.r.t. the standard basis $Int(a)$
(consisting of the fuzzy-intervals), we follow the method in
Section~\ref{sec:gen}, by embedding $B(a)$ into the Boolean cube
$C=2^{[N]}$, where $N=|a|$, and then use the Laurentness property for the
latter.

For a vector $x\in B(a)$, define the subset $[x]\subseteq [N]$
by~\refeq{map[]}. By Proposition~\ref{pr:surj0}, for each TP-function $f$
on $B(a)$, there exists a TP-function $g$ on $C$ such that
  $$
f(x)=g([x]) \qquad\mbox{for all $x\in B(a)$}.
  $$
Recall that the function $g$ constructed in the proof of
Proposition~\ref{pr:surj0} is such that
  $$
g(I)=f(\#(I))+M\eps(I) \qquad\mbox{for each $I\in Int_N$},
  $$
where $\eps(I)$ is the excess of $I$, and $M$ a large positive integer
(cf.~\refeq{quasi}). By Proposition~\ref{pr:laur-cube}, $g([x])$ is
expressed via a tropical Laurent polynomial in variables $g(I)$, where $I$
runs over the set $Int_N$ of intervals in $[N]$. Also the vector $\#(I)$
is a fuzzi-interval for each $I\in Int_N$. Since the convexity preserves
under affine transformations of variables, we obtain that for each vector
$x\in B(a)-Int(a)$, the values $f(x)$ for $f\in\Tscr(B(a))$ are
represented by a piece-wise linear convex function of the arguments
$f(y)$, $y\in Int(a)$, and therefore, $\Tscr(B(a))$ has the Laurentness
property w.r.t. $Int(a)$.

More precisely, $f(x)$ is expressed by the tropical Laurent polynomial
  \begin{equation} \label{eq:laur-box}
  f(x)=\max_{\Fscr\in\Phi_{[x]}} \sum\nolimits_{I\in Int_N}
   h_\Fscr(I)(f(\#(I))+M\eps(I)),
  \end{equation}
where, as before, $\Phi_{S}$ denotes the set of flows in $C$ with the
source set corresponding to $S$.

It remains to estimate the coefficients in the tropical monomials
in~\refeq{laur-box}. The sum concerning a flow $\Fscr\in\Phi_{[x]}$ is of
the form $\sum(\alpha_\Fscr(y)\colon y\in Int(a))+\beta_\Fscr M$ for some
integers $\alpha_\Fscr(y),\beta_\Fscr$. Since $M$ is large,
$\beta_\Fscr\le 0$. Also if $\beta_\Fscr<0$, then the flow $\Fscr$ can be
ignored. So we can consider in~\refeq{laur-box} only the flows $\Fscr$
with $\beta_\Fscr=0$; we call them {\em regular} and denote the set of
these by $\Phi_{[x]}^{{\rm ref}}$.

For $i=1,\ldots,n$, let $\Pi_i$ denote the set of pairs $pq$ such that
$\bar a_{i-1}+1<p\le\bar a_i$ and $q< p-\bar a_{i-1}$, and let $\Gamma_i$
be the triangular sub-grid of $\Gamma_{N,N}$ induced by the vertices
$v_{pq}$ for $pq\in\Pi_i$. Consider a regular flow $\Fscr$ and a path
$P\in\Fscr$; let $h_P(I)$ ($I\in Int_N$), $\alpha_P(y)$ ($y\in Int(a)$)
and $\beta_P$ be the corresponding numbers for $P$. When $P$ makes a turn
at a vertex $v=v_{pq}$, we denote the interval $[p-q+1..p]$ by $I_{pq}$,
and $[p-q+1..p-1]$ by $I'_{pq}$; then the contribution to $\beta_P$ from
$v$ is $\beta_{pq}:=\eps(I'_{pq})-\eps(I_{pq})\le 0$ in case of right
turn, and $\beta_{pq}:=\eps(I_{pq})-\eps(I'_{pq})\ge 0$ in case of left
turn. 
  
  \smallskip
  \noindent
{\bf Claim.} {\em (i) $P$ cannot make any turn within
$\Gamma_1\cup\ldots\cup\Gamma_n$, and (ii) $\beta_P=0$.}

 \smallskip
 \begin{proof}
Suppose $P$ makes a turn at a vertex of some $\Gamma_i$. Let $P$
begin at a block $L_j$ (more precisely, at a source $s_r$ with $r\in
L_j$). The case $i=j$ is impossible (since $[x]\cap L_i$ consists of the
first $x_i$ elements of $L_i$, whence the paths beginning at $L_i$ cannot
make any turn before the ``diagonal'' $\{v_{p'q'}\colon p'\in L_i,\,
q'=p'-\bar a_{i-1}\}$). Therefore, $j>i$. Let $P$ make turns at the
sequence $v_{p(1)q(1)},\ldots,v_{p(k)q(k)}$ of vertices in $\Gamma_i$;
then $p(1)-q(1)>\ldots>p(k)-q(k)>0$. Clearly the edge entering
$v_{p(1)q(1)}$ is horizontal; so $P$ makes right turn at $v_{p(1)q(1)}$.
For $d=1,\ldots,k$, we have $\beta_{p(d)q(d)}=-(p(d)-q(d))$ if $d$ is odd,
and $\beta_{p(d)q(d)}=p(d)-q(d)$ if $d$ is even (taking into account that
$P$ makes right turn at $v_{p(1)q(1)}$, that the size of the head of
$I_{p(d)q(d)}$ is greater by 1 than that of $I'_{p(d)q(d)}$, and the shift
of the head of each of these intervals is equal $p(d)-q(d)$). This implies
that the contribution to $\beta_P$ from the vertices of $P$ within
$\Gamma_i$ is strictly negative.

Also if $P$ makes a turn at an intermediate vertex $v_{pq}$ not in
$\Gamma_1\cup\ldots\cup\Gamma_n$, then $\beta_{pq}=0$ (since one can see
that either $I_{pq},I'_{pq}$ have equal heads, or the shifts of the heads
of these intervals are zero, or both, yielding equal excesses). And if $P$
makes left turn at its beginning vertex, $v_{p1}$ say, then $p$ is the
beginning of a block, and therefore, $\beta_{p1}=0$. Thus, $\beta_P\le 0$.
Now the claim follows from the fact that $\sum(\beta_P\colon P\in\Fscr)$ 
amounts to $\beta_\Fscr=0$.
 \end{proof}

From (i) in the Claim it follows that for each fint $y\in Int(a)$ with 
$|\suppo(y)|=1$, if $\alpha_\Fscr(y)\ne 0$, then there is exactly one 
interval $I\in Int_N$ such that $\#(I)=y$ and $h_\Fscr(I)\ne 0$. A similar
property is trivial if $|\suppo(y)|>1$.  This implies
$\alpha_\Fscr(y)=h_\Fscr(I)$, whence $-1\le\alpha_\Fscr(y)\le 2$. Summing
up the above reasonings, we conclude with the following
   \begin{prop} \label{pr:laur-box}
The set of TP-function $f$ on a box $B(a)$ possesses the tropical
Laurentness property w.r.t. the set $Int(a)$ of fuzzy-intervals. For each
$x\in B(a)-Int(a)$, the value $f(x)$ is expressed as
  $$
  f(x)=\max\nolimits_{\Fscr\in\Phi_{[x]}^{{\rm reg}}} \sum\nolimits_{I\in Int_N}
   h_\Fscr(I)f(\#(I)),
  $$
and all coefficients in the tropical monomials of this expression are in
$\{-1,0,1,2\}$.
   \end{prop}

 \end{document}